\documentclass[11pt]{amsart}
\usepackage{amssymb, amsmath}
\usepackage{amsthm}
\usepackage{amscd}
\usepackage{graphicx}
\newtheorem{theorem}{Theorem}
\newtheorem{lemma}[theorem]{Lemma}

\newtheorem{question}[theorem]{Question}
\newtheorem{corollary}[theorem]{Corollary}
\newtheorem{proposition}[theorem]{Proposition}

\theoremstyle{definition}

\newtheorem*{example}{Example}
\newtheorem*{remark}{Remark}
\newtheorem*{observation}{Observation}

\newtheorem*{acknowledgement}{Acknowledgement}

\title[$2$-complexes not embedding in $\mathbb{R}^4$]
{Fungible obstructions to embedding $2$-complexes}
\author{Grigori Avramidi, T. T$\hat{\mathrm{a}}$m Nguy$\tilde{\hat{\mathrm{e}}}$n-Phan}
\address{Max Planck Institute of Mathematics\\
Bonn\\
Germany}
\email{gavramidi@mpim-bonn.mpg.de}

\address{Karlsruhe Institute of Technology\\
Institute of Algebra and Geometry\\
Karlsruhe\\
Germany}
\email{tam@mpim-bonn.mpg.de}

\def\ra{\rightarrow}

\def\beqa{\begin{eqnarray}}
\def\eeqa{\end{eqnarray}}
\def\beqa{\begin{eqnarray}}
\def\eeqa{\end{eqnarray}}

\DeclareMathOperator{\Lk}{Lk}

\def\x{\hat{x}}

\input xy
\xyoption{all}

\begin{document}

\begin{abstract}
We give new examples of finite, simplicial $2$-complexes that do not PL embed in $\mathbb{R}^4$ and exhibit, for each such complex, a family of PL immersions into $\mathbb{R}^4$ that hide the obstruction to embedding in ``higher and higher order  Milnor invariants". In addition, we show that the embedding obstructions defined by Krushkal in \cite{Krushkal} vanish for our examples. We also answer a question in \cite{AOS}.
%, which, in a different universe, would be a paper itself. 
\end{abstract}
\maketitle

\section{Introduction}\label{intro}
Van Kampen in 1933 gave an elementary method for deciding whether a finite, simplicial $n$-complex does not embed in $\mathbb{R}^{2n}$ (\cite{vanKampenpaper}, \cite{vanKampenerr}). He found a necessary condition for an $n$-complex to embed in $\mathbb{R}^{2n}$ and showed, as applications of this method, that all compact, PL $n$-manifolds PL embed in $\mathbb{R}^{2n}$ and gave examples, in each dimension $n$, of $n$-complexes $X^n$ that do not PL embed in $\mathbb{R}^{2n}$, such as the $n$-skeleton of the $(2n+2)$-dimensional simplex\footnote{When $n=1$, the complex is the complete graph $K_5$ on $5$ vertices.} and the $n$-fold join of $n$ sets of three points. Van Kampen showed that for any generic, PL immersion $X\looparrowright \mathbb{R}^{2n}$, there must be a pair of disjoint simplices of $X$ whose images have nonempty intersection, thus, $X$ cannot PL embed in $\mathbb{R}^{2n}$.  We give a summary of van Kampen's method in Section \ref{vKobstruction}.

In modern language, van Kampen's necessary condition is phrased as ``the van Kampen obstruction of $X$ vanishes". It turns out that when $n\ne 2$, the van Kampen obstruction is the only obstruction to embedding into $\mathbb{R}^{2n}$ for an $n$-complex. Van Kampen, however, did not sucessfully prove the completness of his method in these dimensions since  it requires the Whitney trick when $n>2$, which was developed a decade later. The case $n=1$ follows as a corollary of Kuratowski's theorem on planarity of graphs.  For $n=2$, Freedman, Krushkal and Teichner in 1994 gave examples of finite, simplicial $2$-complexes $K$ which do not  embed in $\mathbb{R}^4$ even though their van Kampen obstructions are zero (\cite{FKT}). The non-embedding mechanism of $K$ is quite similar to how one proves that the Borromean rings do not bound disjoint disks in $\mathbb{D}^4$ even though they have zero pairwise linking numbers. We will review these examples and why they do not embed in $\mathbb R^4$ in Section \ref{FKT}.

In this paper, we exhibit examples of finite, simplicial $2$-complexes $X_k$ that have zero van Kampen obstruction yet do not PL embed in $\mathbb{R}^4$ but for a different reason, which will be addressed below.  The $2$-complexes $X_k$ have many similarities with the $2$-complexes $K$ by Freedman-Krushkal-Teichner given in \cite{FKT}. However, the method in \cite{FKT} does not carry over to show that $X_k$ does not embed, as far as we can see. Indeed, Krushkal (\cite{Krushkal}) formulated a general obstruction in terms of Massey products (and their relation to Milnor invariants) on boundaries of thickenings of a $2$-complex that captures the mechanism preventing the complexes $K$ in \cite{FKT} from embedding in $\mathbb{R}^4$. We will show in Section \ref{Krushkalobstruction} that the obstructions defined by Krushkal vanish for our complexes $X_k$.

More interestingly and surprisingly, 
there are PL immersions $X_k\rightarrow\mathbb{R}^4$ for which the Milnor invariants associated to the images of disjoint triples of simplices of $X_k$ vanish up till arbitrarily large order. We will exhibit these immersions in Section \ref{embeddings}.

Our examples have many similarities with the examples of Freedman-Krushkal-Teichner, so we will briefly recall their examples before getting to the main result of this paper.

\bigskip

\noindent
\textbf{The Freedman-Krushkal-Teichner examples} are obtained by taking two copies of the $2$-skeleton of the $6$-dim simplex (say, one with vertices $x_0, x_1, ..., x_6$ and the other one with vertices $\hat{x}_0, \hat{x}_1, ..., \hat{x}_6$) and ``wedging" them together by means of an edge (e.g. by inserting an edge between $x_6$ and $\hat{x}_6$), then removing the interior of the two simplices $x_4x_5x_6$ and $\hat{x}_4\hat{x}_5\hat{x}_6$ and gluing in a $2$-disk $D^2$ with attaching map $\varphi$ given by an element of the commutator subgroup of $F_2 = \langle x_4x_5x_6, \hat{x}_4\hat{x}_5\hat{x}_6\rangle$.

%It is not difficult to see explicitly that $(K-\partial D^2)$ embeds PL in $\mathbb{R}^4$. With a little bit of thought, one sees that such an embedding can be extended to an immersion $K \rightarrow \mathbb{R}^4$ under which the intersection numbers (counted with orientation) between any pair of disjoint $2$-simplices of $K$ are equal to $0$, which means by definition that the van Kampen obstruction of $K$ vanishes. Yet $K$ does not embed in $\mathbb{R}^4$, as shown in \cite{FKT}. In $K$ there are two disjoint $2$-spheres, namely the tetrahedra $S = x_0x_1x_2x_3$ and $\hat{S} = \x_0\x_1\x_2\x_3$, that together with the special $2$-cell $D^2$ are responsible for preventing $K$ from PL embedding in $\mathbb{R}^4$, in that for any PL embedding $f\colon K -\text{int}(D^2)\hookrightarrow\mathbb{R}^4$, the loop $f(\partial D^2)$ must represent a nontrivial element in $\pi_1(\mathbb{R}^4 -f(S) -f(\hat{S}))$, preventing $f$ from extending to an embedding of $K$. A key ingredient of the proof is a theorem of Stallings (\cite{Stallings}), which in this case implies that $f(\partial D^2)$ must represent a nontrivial element in some $(k+1)$-th nilpotent quotient of $\pi_1(\mathbb{R}^4 -f(S) -f(\hat{S}))$, where $k$ is the last term in the lower central series of $F_2$ that contains the attaching map $\varphi$ of $D^2$. That is how one sees that $f(\partial D^2)$ must represent a nontrivial element in $\pi_1(\mathbb{R}^4 -f(S) -f(\hat{S}))$.

We will explain in more detail in Section \ref{FKT} why $K$ does not PL embed in $\mathbb{R}^4$ but in slightly different wording. The technical difference between our explanation and the proof in \cite{FKT} will be addressed in the first remark at end of Section \ref{noembedd}.

\subsection*{Main result}\label{main result}
We give a family of finite $2$-complexes $X_k$, for $k = 3, 5, 7, ...$,  which have zero van Kampen obstruction but do not embed in $\mathbb{R}^4$. Roughly speaking, the thing that obstructs these $X_k$ from PL embedding in $\mathbb{R}^4$ is the fact that in the free group $F_2 =\langle a, b \rangle$, any element of the form $a^rb^s$, for $r,s \ne 0$ is not a $k$-th power for any $k >1$.

The $2$-complexes $X_k$ are defined as follows. First, let
\begin{itemize}
\item[.] $\Delta_6^2$ and $\hat{\Delta}_6^2$ be two copies of the $2$-skeleton of the $6$-dim simplex,
\item[.] $P_k$ be the ``$k$-fold pseudo-projective-plane", i.e. $P_k$ is obtained by taking a $2$-disk and identifying its boundary via a degree-$k$ covering map of the circle. (So $P_2$ is the classical projective plane.)
\end{itemize}
Let $$X_k = \Delta_6^2\;\#\;\hat{\Delta}_6^2\; \# \;P_k$$ be obtained by 
\begin{itemize}
\item[.] first, inserting an edge between $x_6$ and $\hat{x}_6$, 
\item[.] removing the interior of the $2$-simplices $x_4x_5x_6$ and $\x_4\x_5\x_6$ of the resulting space, and letting $\gamma$ be the loop that is the concatenation of the loop $x_4x_5x_6$ and the loop $\x_4\x_5\x_6$ via the edge $x_6\x_6$ (see Figure 1),
\item[.] attaching the boundary $(P_k - D^2)$, the complement of a disk $D^2$ in $P_k$, along the loop $\gamma$.  
\end{itemize}
A note on base points: in talking about fundamental groups, we will use $x_6\x_6$ as a ``base interval" instead of using a base point.
\begin{figure}[h!]
\centering
\includegraphics[scale=0.5]{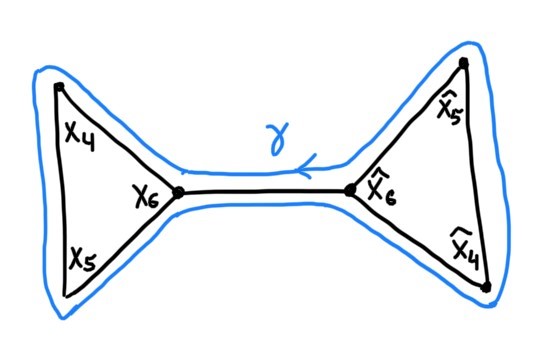}
\caption{The loop $\gamma$.}
\end{figure}

We leave it to the reader to verify that $X_k$ is an iterated connect sum, which explains the above notation. As we will see in Section \ref{obstruction of X_k}, the van Kampen obstruction of $X_k$ is zero when $k$ is odd and we leave it as an exercise for the readers to verify that it is nonzero when $k$ is even.

\begin{theorem}\label{main theorem}
If $k>1$ and an odd integer, then $X_k$ has zero van Kampen obstruction but does not PL embed in $\mathbb{R}^4$.  
\end{theorem}
We will prove Theorem \ref{main theorem} in Section \ref{examples}. Our proof of Theorem \ref{main theorem} depends heavily on the use of links of vertices, which is a PL notion. It is therefore natural to wonder whether $X_k$ \emph{topologically} embeds in $\mathbb{R}^4$. 

\begin{question}
Does $X_k$, for odd $k>1$, topologically embed into $\mathbb{R}^4$?
\end{question}

There are two approaches to showing that a $n$-complex does not embed in $\mathbb{S}^{2n}$. One is to start with a hypothetical embedding and arrive at a contradiction. The other is to start with a particular generic immersion which has some intersections and then show that the intersections do not all go away as one varies the immersion by cooking up an immersion invariant out of them. We used the former approach to show that our examples $X_k$ do not PL embed in $\mathbb{S}^4$. Part of the novelty of these examples is that, in contrast to the examples in \cite{FKT}, %and \cite{SSS},
 they do not seem to be amenable to the later approach. We will illustrate this by exhibiting immersions that hide embedding obstructions from potential invariants in Section \ref{embeddings}. More specifically, we will exhibit,  for each odd $k >1$, various PL immersions $f_n \colon X_k \rightarrow \mathbb{R}^4$ %for which the Milnor invariants associated to all triples of disjoint simplices of $X_k$ vanish up till an arbitrarily high order. In particular, under these immersions $f_n\colon X_k \rightarrow \mathbb{R}^4$, 
under which 
for each $2$-simplex $D^2$ of (a triangulation of) $X_k$ and any subcomplex $S\subset X_k$ that is disjoint from $D^2$ and that is a union of disjoint spheres\footnote{e.g. $S$ is the union of the tetrahedra $x_0x_1x_2x_3$ and $\x_0\x_1\x_2\x_3$.} or wedges of spheres along a simplex\footnote{e.g. the union of the tetrahedra $x_0x_1x_2x_3$ and $x_0x_3x_4x_5$, which are two spheres wedged along the edge $x_0x_3$.}, the restriction of $f$ to $S$ is an embedding and the loop $f(\partial D^2)$ is trivial in the $n$-th nilpotent quotient  of $\pi_1(\mathbb{R}^4 - f(S))$.%  for all $m < n$.

\subsection*{On Section \ref{octahedral}}
The use of the pseudo-projective-plane $P_k$ was motivated by a construction in \cite{AOS}, in which it was asked whether the ``octahedralization" of $P_k$ embeds PL in $\mathbb{R}^4$. Since the proof of Theorem \ref{main theorem} can be modified easily to show that the answer to this question is no when the triangulation of $P_k$ is fine enough (which is not that fine), we will give a proof of this in Section \ref{octahedral}. The specific statement is Proposition \ref{octanotembed}.
%\subsection*{Warning} The complex $\Gamma_k$ is often misunderstood to be one obtained by removing a $2$-simplex in $\Delta_6^2$ and then attaching a $2$-disk $\mathbb{D}^2$ via a degree-$k$ covering map from $\partial\mathbb{D}^2$ to the boundary of the $2$-simplex removed. The latter will always have nonzero van Kampen obstruction whether $k$ is odd or even (as long as $k \ne 0$).

\begin{acknowledgement}
The first author would like to thank the Max Planck Institute of Mathematics, Bonn, and the second author would like to thank the Karlsruhe Institute of Technology for their hospitality and financial support. We would like to thank Shmuel Weinberger for pointing us to Baumslag's paper \cite{Baumslag} and Slava Krushkal for comments on an earlier version of the paper. 
\end{acknowledgement}

Unless otherwise stated, all immersions or embeddings in this paper are PL and \emph{generic}\footnote{The only singularities are double points.}.
\section{On the van Kampen obstruction}\label{vKobstruction}
We will begin this section by recalling van Kampen's method, and  then give a proof that the $2$-skeleton $\Delta_6^2$ of the $6$-dim simplex does not embed in $\mathbb{R}^4$. This is not exactly the same as van Kampen's proof but it is very similar in spirit and is more relevant to the proof of the main result. For convenience, we will prove that $\Delta_6^2$ does not embed in $\mathbb{S}^4$ instead of in $\mathbb{R}^4$.

\subsection{Van Kampen's method}
Let $X^n$ be a finite, simplicial $n$-complex $X^n$.   For any immersion $f\colon X^n \rightarrow \mathbb{R}^{2n}$, van Kampen formed a vector $V_f$, which we will call the \emph{van Kampen vector}, whose components correspond to pairs of disjoint $n$-simplices in $X$. Each component of $V_f$ is equal to the intersection number of the images of the corresponding pair of simplices, which takes value either in $\mathbb{Z}$ if one assigns to each simplex of $X$ an orientation or in $\mathbb{Z}/2\mathbb{Z}$ if one ignores orientations. Thus, if $f$ is an embedding, then $V_f$ is the zero vector. A different generic immersion $g \colon X\rightarrow \mathbb{R}^{2n}$ might give a different vector $V_g$, e.g. if $g$ is obtained by modifying $f$ by pushing the interior of one $n$-simplex $F$ of $X$ across a $(n-1)$-simplex $E$, as illustrated in Figure 2.
\begin{figure}[h!]
\centering
\includegraphics[scale=0.4]{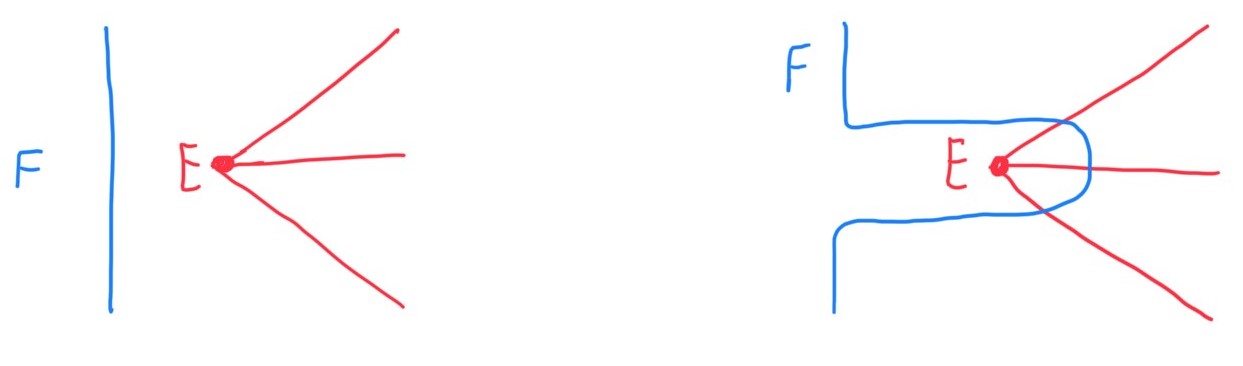}
\caption{A \emph{finger move}, illustrated with graphs.}
\end{figure}

In this case, the difference between $(V_f - V_g)$ is a vector $W_{F, E}$ whose non-zero components correspond to precisely the pairs $(F, F')$ for all $n$-simplices $F'$ that contain the $(n-1)$-simplex $E$ and that are disjoint from $F$.

Van Kampen proved that, in general, for any two  immersions $f, g\colon X^n \rightarrow \mathbb{R}^{2n}$, the difference $(V_f - V_g)$ is a sum of such vectors $W_{F_i,E_j}$, where $F_i$ is an $n$-simplex and $E_j$ is an $(n-1)$-simplex of $X$. (Note that $W_{F_i,E_j}$ depends on $X$ only and does not depend of an immersion of $X$.) Thus, if $X$ embeds in $\mathbb{R}^{2n}$, then  for any immersion $f\colon X\rightarrow \mathbb{R}^{2n}$, the vector  $V_f$ must be a linear combination of vectors $W_{F_i,E_j}$. In other words, if the vector $V_f$, for some immersion $f\colon X\rightarrow \mathbb{R}^4$, is not a linear combination of vectors $W_{F_i,E_j}$, then $X$ does not PL embed in $\mathbb{R}^4$. This is the mechanism van Kampen used to show that the two $n$-complexes mentioned at the beginning of the paper do not embed in $\mathbb{R}^{2n}$.  In modern language, we say that the \emph{van Kampen obstruction} of $X$ is zero if $V_f$ is a linear combination of vectors $W_{F_i,E_j}$ and is nonzero otherwise. 

Van Kampen also showed that when $n\geq 3$, intersections between adjacent $n$-simplices or self-intersections of an $n$-simplex of $X$ can be removed by perturbing the immersion. This means, with the aid of Whitney disks introduced in the following decade, that when $n\geq 3$, $X^n$ embeds in $\mathbb{R}^{2n}$ if and only if $V_f$ (for some PL immersion $f\colon X\rightarrow\mathbb{R}^{2n}$) is a linear combination of $W_{F_i,E_j}$. In other words, when $n\geq 3$, $X^n$ embeds in $\mathbb{R}^{2n}$ if and only if the van Kampen obstruction of $X$ is zero.

\subsection*{Triangulations versus regular regions}
Note that whether the van Kampen obstruction of a complex $X$ is equal to zero does not depend on the triangulation of $X$. Of course, for different subdivisions of $X$, we have different van Kampen vectors $V_f$ of different sizes. However, what really matters is the connected components of the non-singular subset of $X$, i.e. the set of manifold points of $X$, and how they fit together at the singular set. This is because of the following reason. 

We will call each connected component of the non-singular subset of $X$ a \emph{regular region}. For a PL triangulation of $X$ (that is compatible with the original simplicial structure of $X$), for each regular region we pick an $n$-simplex that represents that region. Then, for any immersion $f\colon X \rightarrow \mathbb{R}^4$, there may be many intersections of simplices that are not representatives, but we can modify $f$ by pushing each such intersection between two non-representative simplices so that it moves from one simplex to another along a track within the regular region containing the simplex until it is in a representative simplex. Thus, we can always modify $f$ with finger moves to obtain an immersion for which all intersections are between representative simplices. 

Hence, the van Kampen vector can practically be taken to be such that the number of its  entries is equal to the number of pairs of non-adjacent regular regions of $X$.

\subsection{The complex $\Delta_6^2$ does not embed in $\mathbb{S}^4$}\label{wlog}
Let the vertices of $\Delta_6^2$ be called $x_0, x_1, ..., x_6$ as before. Suppose for contradiction that there is a PL embedding $f\colon\Delta_6^2\rightarrow\mathbb{S}^4$. We can assume without loss of generality (the skeptical readers can see Section \ref{PLversusL} for an explanation) that  $x_0$ is mapped to the North Pole of $\mathbb{S}^4$ and 
\begin{itemize}
\item[*] the boundary of the $1$-neighborhood of $x_0$ (which is the subcomplex $K_6 \subset\Delta_6^2$ that is the complete graph  on $x_1, x_2, ..., x_6$) is mapped to the equator $\mathbb{S}^3$, and the image of the $1$-neighborhood of $x_0$ is the cone on $f(K_6)$,
\item[**] all other $2$-simplices of $\Delta_6^2$ (i.e. those whose edges belong to $K_6$) are contained in the closed Southern hemisphere, which is $\mathbb{D}^4$.  
\end{itemize}
By (**), the complex $\Delta_6^2$ embeds in $\mathbb{S}^4$ only if one can fill in all the $3$-cycles in $K_6 \hookrightarrow \mathbb{S}^3$ with disjoint $2$-disks in $\mathbb{D}^4$. Since  each $3$-cycle in $K_6$ bounds a $2$-simplex in $\Delta_6^2$, it follows that any pair of disjoint $3$-cycles $c_1$ and $c_2$ in $K_6$ must have linking number $\Lk(c_1, c_2) =0$. However, this is not possible by a theorem of Conway and Gordon, which says that $K_6$ cannot \emph{linklessly} embed in $\mathbb{S}^3$.  

\begin{theorem}[Conway-Gordon \cite{ConwayGordon}]
Any embedding of the complete graph $K_6$ on six vertices into $\mathbb{S}^3$ contains a pair of disjoint cycles with nonzero linking number. 
\end{theorem}
\begin{proof}[Proof idea]
The key idea is  to take, for each embedding $\varphi$ of $K_6$, the sum of the linking numbers of all pairs of disjoint cycles in $K_6$ 
\[\mathcal{L}(\varphi) = \sum_{\text{all disjoint cycles of}\; K_6}  \Lk (c_1, c_2)\]
and show that it is congruent to 1 (mod 2).

To show that this sum is independent modulo 2 of embeddings of $K_6$, note that given any two embeddings of $K_6$, one embedding can be obtained by performing some crossing changes on the other. Thus, if we can show that the above sum changes by an even number after any crossing change, and if we compute it for a particular embedding (any would do) and see that it is odd, then we will be done. The first is a simple counting argument and the second is an easy computation, both of which are good exercises for the readers. (See also the example below.)   
\end{proof}

As an example, consider the embedding $\alpha\colon  K_6\rightarrow\mathbb{S}^3$ given in Figure 3(a). 
\begin{figure}[h!]
\centering
\includegraphics[scale=0.5]{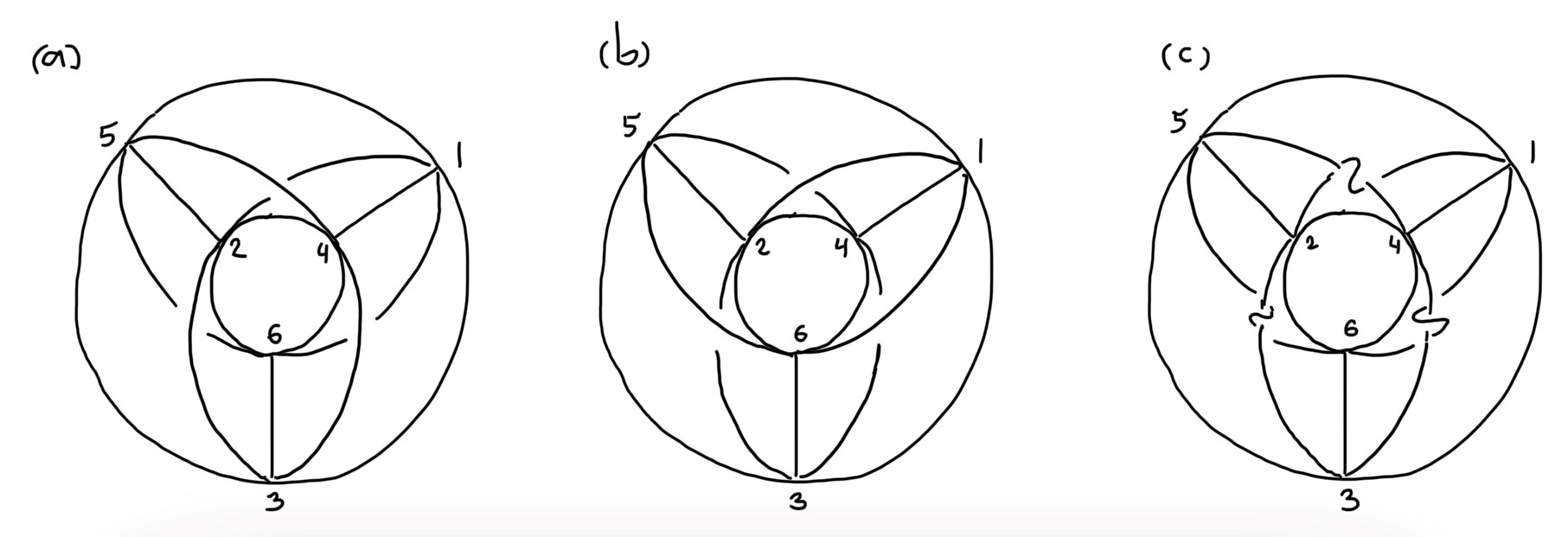}
\caption{Different embeddings of $K_6$ into $\mathbb{S}^3$.}
\end{figure}
The only pair of disjoint $3$-cycles with non-zero linking number, which is $\pm 1$ in this case, is $(x_1x_2x_3, x_4x_5x_6)$. With respect to a choice of orientations\footnote{Choices of orientations are made on $\mathbb{D}^4$ and on the $2$-simplices of $\Delta_6^2$, which will induce orientations on each $3$-cycle of $K_6$.}, suppose that the linking number $$\Lk (x_1x_2x_3, x_4x_5x_6) =1.$$ This means that for this embedding  $f\colon\Delta_6^2\rightarrow\mathbb{S}^4$, the van Kampen vector, as defined at the beginning of this section, is 
$$V_f = (1, 0, 0, ..., 0),$$ 
where the component that is equal to $1$ corresponds to the pair of $2$-simplices $x_1x_2x_3$ and  $x_4x_5x_6$. 
(Side note: van Kampen showed that each of the vectors $W_{F_iE_j}$ has components summing to an even number, therefore any PL immersion $g\colon \Delta_6^2\rightarrow\mathbb{S}^4$ must have nonzero van Kampen vector, which implies by van Kampen's method that $\Delta_6^2$ does not embed in $\mathbb{S}^4$. This is van Kampen's original proof of this fact.)

It turns out that the van Kampen vector of $\Delta_6^2$ can be equal to $(k,0,0,..., 0)$ for any odd integer $k$. This is a consequence of  the general fact that ``the van Kampen obstruction has order $2$", which means that, for a finite, simplicial $n$-complex $X^n$, if $V_f$ is the van Kampen vector associated with a PL immersion of a complex $f\colon X^n \rightarrow \mathbb{R}^{2n}$, then $2V_f$ is a linear combinations of finger move vectors $W_{F_i, E_j}$.

\subsection{The van Kampen obstruction has order $2$}
That the van Kampen obstruction has order $2$ follows from the fact that there is a PL immersion\footnote{An easy choice of $g$ is the composition of $f$ with a reflection of $\mathbb{R}^{2n}$.} $g\colon X^n \rightarrow \mathbb{R}^{2n}$ with $V_g = -V_f$. We will illustrate this via an example with the complex $\Delta_6^2$.

Consider the embedding $\beta\colon K_6\rightarrow \mathbb{S}^3$ given in Figure 3(b). With respect to the same choice of orientations from before, we have  $$\Lk (x_1x_2x_3, x_4x_5x_6) = -1,$$
and the linking numbers of all other pairs of disjoint $3$-cycles are equal to zero. That is, the van Kampen vector $$V_g = (-1, 0, 0, ..., 0).$$ 
By van Kampen's theorem, $(V_g - V_f)$ is a equal to a sum of ``finger move" vectors $W_{F_i, E_j}$, i.e. 
\[V_g - V_f = \sum_{(i,j)\in \Lambda} W_{F_i,E_j}. \]

An application of this is that if we apply the \emph{negatives} of the exact same collection of finger moves to $f$, then we will get a PL immersion $h$ with van Kampen vector 
\[V_h = V_f - \sum_{(i,j)\in \Lambda} W_{F_i,E_j} = V_f - (V_g - V_f ) = (3, 0, 0, ..., 0)\] 
If we keep applying the same collection of finger moves, we will get a PL immersion of $\Delta_6^2$ with van Kampen vector equal to $(k, 0, 0, ..., 0)$, for any odd integer $k$.

\subsection*{On Figure 3(c)} There is a very similar order-2 phenomenon that happens for embeddings of $K_6\rightarrow \mathbb{S}^3$. Note that the embedding $\beta \colon K_6 \rightarrow \mathbb{S}^3$ (given by Figure 3(b)) is obtained by composing the embedding $\alpha\colon K_6 \rightarrow \mathbb{S}^3$ (given by Figure 3(a)) with reflection. As expected, going from $\alpha$ to $\beta$, all the linking numbers between pairs of disjoint $3$-cycles change to their negatives. 

Alternatively, $\beta$ is obtained from $\alpha$ by doing $3$ crossing changes. We leave it to the readers to identify these crossing changes. Now, if we apply the negatives of these three crossing changes to Figure 3(a), we get Figure 3(c), which gives an embedding $\gamma \colon K_6 \rightarrow \mathbb{S}^3$. Note that with respect to $\gamma$ and the choices of orientations from before, the linking number $$\Lk (x_1x_2x_3, x_4x_5x_6) = 3,$$
and the linking numbers of all other pairs of disjoint $3$-cycles are equal to zero. In a similar manner, ``$3$" can be replaced by any odd integer if we apply enough times these three crossing changes.

\subsection{PL vs simplicial embeddings}\label{PLversusL}

In this subsection, we elaborate on why we can assume that a PL embedding of a $2$-complex in $\mathbb S^4$ is in the standard form described at the beginning of Section \ref{wlog}. The main technical point is to show the following.

\begin{proposition}
Any PL embedding of a $2$-complex $X$ in $\mathbb S^4$ can be replaced by a possibly different PL embedding for which the triangulation of $X$ extends to a triangulation of $\mathbb S^4$.
\end{proposition}

When $X=\Delta_6^2$ and $f \colon X \rightarrow \mathbb{S}^4$ is a simplicial embedding, by looking at the star $St_{S^4}(f(x_0))$, which is a ball that can be identified with the northern hemisphere, we see that $f$ satisfies condition (*). Since $f$ is simplicial, the images of the simplices of $X$ that do not contain $x_0$ cannot intersect with the open star of $f(x_0)$ (which is the open Northern hemisphere), so (**) is satisfied.

%think of $f(x_0)$ as the north pole, the star $St_{S^4}(f(x_0))$ as the northern hemisphere and $Lk_{S^4}(f(x_0))$ as the equator 
%the $1$-neighborhood of $X$ is the cone on $K_6$ 
%this embedding has the standard form described in Section \ref{wlog}. 

\begin{proof}
Suppose a $2$-complex $X$ embeds in $\mathbb S^4$ via a PL embedding $f:X\hookrightarrow \mathbb S^4$. We would like to extend the triangulation of $X$ to a triangulation of $\mathbb S^4$. However, this is not always possible. For instance, if there is an interior point $x$ of a $2$-cell in $X$ such that $Lk_X(x)$ is a knotted circle in $Lk_{S^4}(x)=\mathbb S^3$, 
\begin{figure}[h!]
\centering
\includegraphics[scale=0.20]{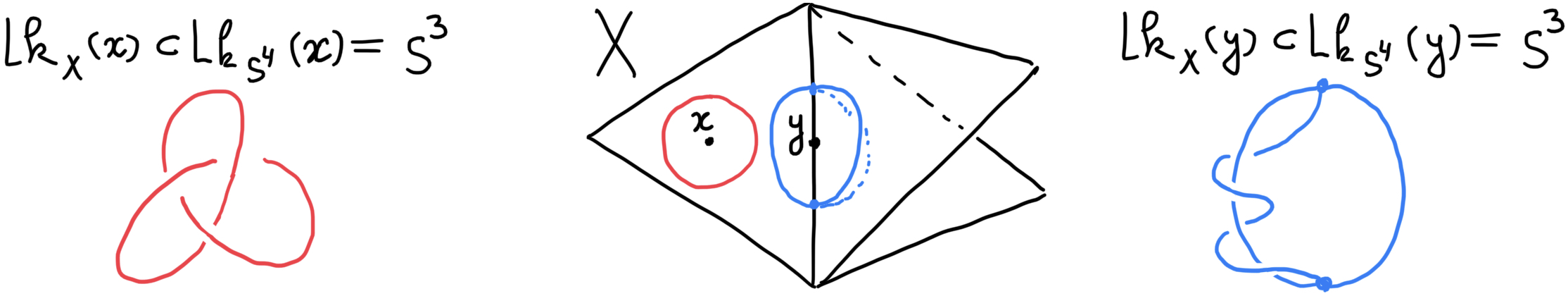}
\caption{Local knottedness in the interior of a $2$-simplex (left) and edge (right).}
\end{figure}
then the triangulation cannot possibly extend.\footnote{For an explicit example, PL triangulate a non-trivial knot $K \hookrightarrow \mathbb S^3$ and suspend to get a PL embedding $\mathbb S^2\hookrightarrow \mathbb S^4$ that is locally knotted at the two suspension points. If we now take a new triangulation of $\mathbb S^2$ in which the suspension points are interior points, then this new triangulation on $\mathbb S^2$ does not extend to $\mathbb S^4$.} To keep track of this phenomenon, we say that the embedding $f$ is {\it locally knotted} at the point $x\in X$ if $Lk_X(x)$ is a suspension of a finite set of points $F$ while the embedding $Lk_X(x)\hookrightarrow Lk_{\mathbb S^4}(x)=\mathbb S^3$ is not the suspension of an embedding of $F$ in $\mathbb S^2$. Figure 4 illustrates local knottedness.% in the interior of a $2$-simplex and on an edge.

If the embedding is simplicial, i.e. if $X$ is embedded as a subcomplex of a triangulation of $\mathbb S^4$, then the embedding can only be locally knotted at the vertices. %For a PL embedding, a subdivision of $X$ embeds as a subcomplex of a triangulation of $S^4$, so $f$ is locally knotted at finitely many points, but these points may not be vertices of $X$. 
Conversely, Akin showed (Theorem 2 of \cite{akin}) that if all non-vertex points of $X$ are locally unknotted, then the triangulation of $X$ extends to a triangulation of $\mathbb S^4$.

Now, observe that we can modify the PL embedding $f$ to a new PL embedding that is only locally knotted at the vertices. To see this, first note that since $f$ is PL, the number of locally knotted points is finite (since they occur as vertices of some subdivision of $X$). Suppose there is a locally knotted non-vertex point $x$ lying in some simplex $\sigma$ of $X$. 
\begin{figure}[h!]
\centering
\includegraphics[scale=0.17]{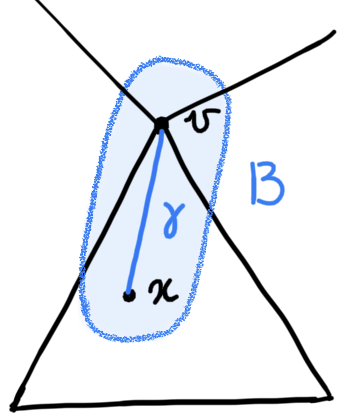}
\caption{The neighborhood $B$.}
\end{figure}
Connect $x$ to a vertex $v$ of $\sigma$ by a line $\gamma$ lying entirely in $\sigma$. Let $B$ be a small regular neighborhood of $\gamma$ in $\mathbb S^4$ (see Figure 5). 

Then $B$ is a PL $4$-ball. The boundary $\partial B=\mathbb S^3$ intersects $X$ in a graph PL homeomorphic to $Lk_X(v)$, and $B$ intersects $X$ in a $2$-complex PL homeomorphic to the cone on $Lk_X(v)$. Now, we modify the embedding $f$ on $B$ by replacing $f\mid_B$ with the cone on $f\mid_{\partial B}$. The resulting PL embedding $f':X\hookrightarrow \mathbb S^4$ has fewer `non-vertex' locally knotted points than $f$, so repeating the construction finitely many times produces a PL embedding $f'':X\hookrightarrow \mathbb S^4$ that is only locally knotted at vertices. For this embedding the triangulation of $X$ extends to $\mathbb S^4$ by Akin's theorem. 
\end{proof}

\begin{remark}
We can further arrange for the simplicial embedding $f$ so that for each vertex $v\in X\in \mathbb S^4$ the $1$-neighborhood (or link) of $v$ in $\mathbb S^4$ intersects $X$ in the $1$-neighborhood (or link) of $v$ in $X$, i.e. 
\begin{eqnarray*}
St_X(v)&=&St_{\mathbb S^4}(v)\cap X,\mbox{ and}\\
Lk_X(v)&=&Lk_{\mathbb S^4}(v)\cap X.
\end{eqnarray*}
This can be achieved by taking partial barycentric 
\begin{figure}[h!]
\centering
\includegraphics[scale=0.17]{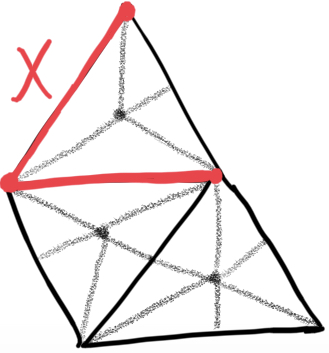}
\caption{Partial barycentric subdivision.}
\end{figure}
subdivisions of $\mathbb S^4$ (subdividing $\mathbb S^4$ but not $X$, see Figure 6) if necessary. 

%we get a triangulation of $\mathbb S^4$ containing $X$ as a subcomplex such that for a vertex $v\in X\in \mathbb S^4$ the closed star (or link) of $v$ in $\mathbb S^4$ intersects $X$ in the closed star (or link) of $v$ in $X$, i.e. 
%\begin{eqnarray*}
%St_X(v)&=&St_{\mathbb S^4}(v)\cap X,\mbox{ and}\\
%Lk_X(v)&=&Lk_{\mathbb S^4}(v)\cap X.
%\end{eqnarray*}

A similar statement holds when $v$ is replaced by an edge $e$ whose $1$-neighborhood $N_1(e,X)$ in $X$ satisfies the property that every triangle made of three edges in $N_1(e,X)$ must bound a $2$-simplex in $X$. In such a case, by taking subdivision on $\mathbb{S}^4$, one can arrange so that the $1$-neighborhood of $e$ in $\mathbb{S}^4$ is a ball $B$ whose intersection with $X$ is the $1$-neighborhood of $e$ in $X$, and in addition, the intersection $\partial B \cap X$ is the boundary of the $1$-neighborhood of $e$ in $X$. 

\end{remark}

\section{Some $2$-complexes with zero van Kampen obstructions (that do not embed)}\label{examples}
In this section we will explain why the van Kampen obstruction of $X_k$, for $k$ odd, is zero and why it does not PL embed in $\mathbb{S}^4$. The proof of Theorem \ref{main theorem} is given at the end of this section. But before talking about $X_k$, we will warm up with the Freedman-Krushkal-Teichner examples. There are many similarities between these two families of $2$-complexes, so we will first make a few simplications that apply to both.
\subsection{As a matter of simplification}\label{simplify} Let $\bowtie$ be the $2$-complex obtained by taking two copies $\Delta_6^2$ and $\hat{\Delta}_6^2$ of the $2$-skeleton of the $6$-simplex, inserting an edge connecting vertex $x_6$ to $\x_6$ and then removing the interior of the faces $x_4x_5x_6$ and $\x_4\x_5\x_6$.  Let $Y$ be a $2$-complex that contains $\bowtie$ as a subcomplex. Let us assume also that any cell in $Y$ that does not belong to $\bowtie$ does not contain $x_0$ or $\x_0$. Note that both $X_k$ and the Freedman-Krushkal-Teichner complexes are examples of $Y$.

We now set things up to prove Theorem \ref{main theorem} in Section \ref{proofoftheorem}, which we will do by contradiction. So suppose that $Y$ embeds in $\mathbb{S}^4$ and for simplicity in notations we will treat $Y$ as a subset of $\mathbb{S}^4$. Then if we attach to $Y$ a square along $3$ of its sides to $x_0x_6$, $x_6\x_6$ and $\x_6\x_0$ (see Figure 7),
\begin{figure}[h!]
\centering
\includegraphics[scale=0.47]{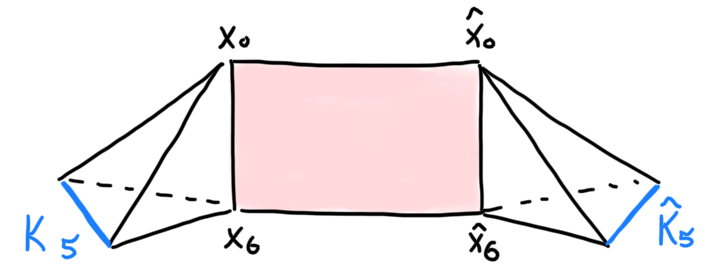}
\caption{The $1$-neighborhood of $x_0\x_0$.}
\end{figure}
then we will obtain a complex $Z$ that also embeds\footnote{If the square intersects nontrivially with any simplex of $Y$, we can always push the intersections off the square through the free edge $x_0\x_0$.} in $\mathbb{S}^4$. The complex $Z$ has a free edge, which is $x_0\x_0$, and it is important that we do not collapse it. 

By the discussion at the end of the previous section, we can assume without loss of generality that the restriction of the embedding to the simplicial $1$-neighborhood of $x_0\x_0$ in $Y$ is simplicial. Thus, the simplicial $1$-neighborhood of $x_0\x_0$ is contained in a closed, PL $4$-ball $B$, and the rest of $Y$ lies in the complement of $B$. So, the complement of the interior of $B$ is a closed, PL ball $\mathbb{D}^4$ which contains all the simplices of $Y$ with vertices in $\{x_1,..,x_6, \x_1,...,\x_6\}$. Moreover, the vertices $x_1,..,x_6, \x_1,...,\x_6$ and the edges connecting any two of them are in $\mathbb{S}^3 = \partial\mathbb{D}^4$, i.e. in $\mathbb{S}^3$ one sees two embedded complete graphs on six vertices, which we denote as $K_6$ and $\hat{K}_6$, and an edge $x_6\x_6$. Furthermore, thinking of $\mathbb{S}^3$ as the boundary of the $1$-neighborhood of the edge $x_0\x_0$, we see that $\mathbb{S}^3$ decomposes as a union 
\[\mathbb{D}^3 \;\cup\; (\mathbb{S}^2\times[0,1]) \;\cup \; \hat{\mathbb{D}}^3,
\]
where $\mathbb{D}^3$ and $\hat{\mathbb{D}}^3$ are disjoint PL $3$-balls in $\mathbb{S}^3$. We observe that $\mathbb{D}^3$ (respectively, $\hat{\mathbb{D}}^3$) contains $K_6$ (respectively, $\hat{K}_6$) and the edge $x_6\x_6$ lies in the tube $\mathbb{S}^2\times[0,1]$ (in fact, in a standard way). This observation will simplify a lot of things later so we summarize it as follows.

\begin{observation} Any PL embedding into $\mathbb{S}^4$ of a complex $Y$ induces a PL embedding $f$ into $\mathbb{D}^4$ of the subcomplex of $Y$ consisting of all simplices that do not contain $x_0$ or $\x_0$. Furthermore, the intersection of the image of $f$ with $\partial\mathbb{D}^4 = \mathbb{S}^3= \mathbb{D}^3 \;\cup\; (\mathbb{S}^2\times[0,1]) \;\cup \; \hat{\mathbb{D}}^3$ consists of an embedded $K_6 \subset \mathbb{D}^3$, an embedded $\hat{K}_6 \subset \hat{\mathbb{D}}^3$ and a standardly embedded edge $x_6\x_6$ in $\mathbb{S}^2\times[0,1]$. (See Fig. 8.)
\end{observation}
\begin{figure}[h!]
\centering
\includegraphics[scale=0.3]{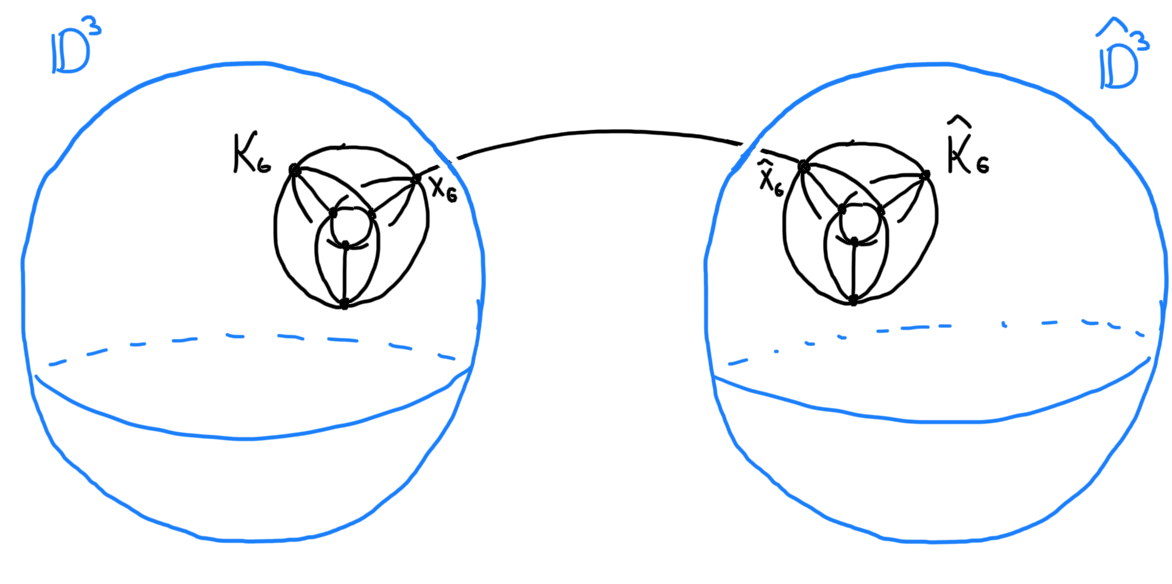}
\caption{View from $x_0\x_0$ of the the boundary of its $1$-neighborhood.}
\end{figure}

It follows from this observation that the fundamental group of the complement $$C\colon=\mathbb{S}^3 - (x_1x_2x_3 \cup \x_1\x_2\x_3)$$ of the two embedded loops $x_1x_2x_3$ and $\x_1\x_2\x_3$ decomposes as a free product
\[\pi_1(C) = Z*\hat{Z},\]
where $Z = \pi_1(\mathbb{D}^3 - x_1x_2x_3)$ and $\hat{Z} = \pi_1(\hat{\mathbb{D}}^3 - \x_1\x_2\x_3)$. 

The abelianization map on each factor of $Z*\hat{Z}$ induces a homomorphism to the free group $F_2$ on two generators $a$ and $b$
\[ \eta\colon Z * \hat{Z} \longrightarrow \mathbb{Z}* \mathbb{Z} = F_2 = \langle a, b\rangle.\]
Note that $a$ is the image of the element of $Z$ represented by a small loop that goes around $x_1x_2x_3$, and $b$ is the image of the element of $\hat{Z}$ represented by a small loop that goes around $\x_1\x_2\x_3$.

\subsection*{Useful fact} Let $W$ be the complement of the two $2$-simplices $x_1x_2x_3$ and  $\x_1\x_2\x_3$ in $\mathbb{D}^4$. Then, the inclusion $C\hookrightarrow W$ induces the following isomorphisms of nilpotent quotients\footnote{For a group $G$, let $G_{1} = G$ and let $G_{i+1} = [G, G_{i}]$ for $i = 1, 2, ...$. Then $G/G_n$ is the $n$-th nilpotent quotient of $G$.} 
\[\pi_1(W)/(\pi_1(W)_{n})\;\cong \;\pi_1(C)/(\pi_1(C)_{n}), \]
for all $n = 1,2,3,...$, by a theorem of Stallings (Theorem 5.1 in \cite{Stallings}) and Alexander duality, as follows. Stallings's theorem says that if a map $C\rightarrow W$ between cell-complexes induces an isomorphism $H_1(C) \rightarrow H_1(W)$ and an epimorphism $H_2(C)\rightarrow H_2(W)$, then it induces an isomorphism on all nilpotent quotients $\pi_1(C)/(\pi_1(C)_{n})\;\cong \;\pi_1(W)/(\pi_1(W)_{n})$. Now, in our situation, Alexander duality implies that $H_2(W) = 0$ (thus, the induced map on second homology is onto) and the inclusion $C\hookrightarrow W$ induces an isomorphism $H_1(C) \cong H_1(W)$.  Therefore, the hypotheses in Stallings' theorem are satisfied, so the conclusion follows.

\subsection{The examples of Freedman-Krushkal-Teichner}\label{FKT}
The $2$-complex $K$ given by Freedman-Krushkal-Teichner is obtained by attaching to $\bowtie$ a $2$-cell $D^2$ via the attaching map $\psi$ given by a nontrivial element $w$ in the commutator subgroup of the free group $\langle x_4x_5x_6, \x_4\x_5\x_6\rangle$ on two generators which are the loops $x_4x_5x_6$ and $\x_4\x_5\x_6$ (we would like to remind the readers that in talking about fundamental groups, we will, for convenience, use a ``base interval", i.e. $x_6\x_6$, instead of using a base point.) That is,
\[K \;= \;\bowtie \;\cup_\psi \;D^2.\]

%We will first explain why $K$ does not PL embed in $\mathbb{S}^4$ and then explain why the van Kampen obstruction of $K$ is zero.

\subsubsection{$K$ does not embed in $\mathbb{S}^4$}
Suppose that there is PL embedding $f \colon K \rightarrow\mathbb{S}^4$. Since $K$ contains a copy of $\bowtie$ and the only other cell in $K$ that is not included in $\bowtie$ is $D^2$, which does not contain $x_0$ or $\x_0$, the above observation applies. It follows that $f\circ\psi$ must be trivial in $\pi_1(W)$. Therefore, $f\circ\psi$ must be trivial in the $n$-th nilpotent quotient $\pi_1(W)/(\pi_1(W)_{n})$ of $\pi_1(W)$ for all $n \in\mathbb{Z}^+$, and thus by the above Fact, $f\circ\psi$ must be trivial in $\pi_1(C)/(\pi_1(C)_{n})$ for all $n \in\mathbb{Z}^+$.

Now, the homomorphism $\eta \colon \pi_1(C) \longrightarrow F_2 =\langle a, b\rangle$ 
discussed above induces homomorphisms $$\eta_n \colon  \pi_1(C)/(\pi_1(C)_{n})\longrightarrow F_2/(F_2)_n$$ 
on nilpotent quotients. Hence, the image of $f\circ \psi$ under $\eta_n$ must be trivial in the $n$-th nilpotent quotient of $F_2$ for all $n\in \mathbb{Z}^+$. However, we claim that this is impossible for the following reason.

By the Conway-Gordon theorem, there must be a pair of disjoint $3$-cycles in $K_6$ with mod $2$ linking number (with respect to the embedding given by $f$) equal to $1$. Since all $3$-cycles in $K_6$, except for $x_4x_5x_6$, bound disjoint disks in $\mathbb{D}^4$, it follows that $\Lk(x_1x_2x_3, x_4x_5x_6) = 1$ mod $2$. Similarly, we see that $\Lk(\x_1\x_2\x_3, \x_4\x_5\x_6) = 1$ mod $2$. Since $K_6$ and $\hat{K}_6$ are contained in disjoint $3$-balls in $\mathbb{S}^3$, it follows that in $\pi_1(C)$, we have
\[ \eta (x_4x_5x_6) = a^r, \quad \eta (\x_4\x_5\x_6) = b^s,\]
for some odd integers $r$ and $s$. In particular, $r, s \ne 0$. Hence, the image of $f\circ \psi$ under $\eta$ is $w(a^r,b^s)$. Since the word $w(a,b)$ is non-trivial, it is non-trivial in some nilpotent quotient $F_2/(F_2)_n$. Stalling's rational theorem (7.3 in \cite{Stallings}) implies that the map $F_2\ra F_2$ for which $a\mapsto a^r$ and $b\mapsto b^s$ induces inclusions of nilpotent quotients $F_2/(F_2)_n\hookrightarrow F_2/(F_2)_n$, so the word $w(a^r,b^s)$ is also non-trivial in $F_2/(F_2)_n$.
%a commutator in $a^r$ and $b^s$, and therefore is nontrivial in the $n$-th nilpotent quotient $F_2/(F_2)_n $ for $n$ large enough. 
Thus, $K$ does not PL embed in $\mathbb{S}^4$.

\subsubsection{$K$ has zero van Kampen obstruction}\label{obstruction of K}
To see that the van Kampen obstruction of $K$ is zero, we exhibit an immersion $ f\colon K\rightarrow \mathbb{S}^4$ for which the van Kampen vector $V_f$ is equal to $0$. 

First, we embed $\bowtie$ into $\mathbb{S}^4$ the obvious way. For the sake of concreteness, we can, for example, when $k=3$ use the embedding in Figure 3(c) to embed $K_6$ and $\hat{K}_6$ in disjoint balls in $\mathbb{S}^3$, and then connect $x_6$ and $\x_6$ by a straight line segment outside these balls. (For general $k$, we just need to add more ``twists" to Figure 3(c) so that $\Lk(123,456) =k$.) Thinking of $\mathbb{S}^3$ as the equator of $\mathbb{S}^4$, we can, in the Northern hemisphere, cone off $K_6$ to a point $f(x_0)$ and cone off $\hat{K}_6$ to a point $f(\x_0)$ in such a way that the two cones are disjoint. Then we fill in all $3$-cycles in $K_6$ and $\hat{K}_6$, except for $x_4x_5x_6$ and $\x_4\x_5\x_6$ with disks in the Southern hemisphere $\mathbb{D}^4$.

Up to this point, we see that the components of the van Kampen vector $V_f$ corresponding to disjoint pairs of $2$-simplices $\sigma_1, \sigma_2$ in $\bowtie$ are equal to $0$ for the following reason. If either $\sigma_1$ or $\sigma_2$ contains $x_0$ or $\x_0$, then their images are clearly disjoint. Otherwise, $\sigma_1$ and $\sigma_2$ bound $3$-cycles in $K_6$ or $\hat{K}_6$ that are not $x_4x_5x_6$ or $\x_4\x_5\x_6$, so the pair of $3$-cycles must have zero linking number, which means that the intersection number of the images of $\sigma_1$ and $\sigma_2$ under $f$ must be zero. 

Next, we map in the left over cell $D^2$ of $K$ into the Southern hemisphere $\mathbb{D}^4$. Note that for this immersion of $D^2$, there are only two components of the van Kampen vector that we need to worry about, namely those corresponding to the pair $(D^2, x_1x_2x_3)$ and the pair $(D^2, \x_1\x_2\x_3)$. Since the attaching map of $D^2$ is a commutator in $\pi_1(\mathbb{S}^3 -  x_1x_2x_3 -  \x_1\x_2\x_3)$, it must be trivial both in $H_1(\mathbb{S}^3 -  x_1x_2x_3)$ and $H_1(\mathbb{S}^3 -  \x_1\x_2\x_3)$, which implies that both the linking numbers $\Lk(\partial D^2, x_1x_2x_3)$ and $\Lk(\partial D^2, \x_1\x_2\x_3)$ are equal to zero. Therefore, the van Kampen vector is equal to zero.

\subsection{The complex $X_k$}\label{proofoftheorem}
The complex $X_k$, as defined in Section \ref{intro}, is obtained by attaching to $\bowtie$ a copy of $(P_k-D^2)$ via an attaching map $\varphi\colon \partial (P_k-D^2) \rightarrow\; \bowtie$ given by the loop $\gamma = x_4x_5x_6 * \x_4\x_5\x_6$.
\[X_k \;= \;\bowtie \;\cup_\varphi \;(P_k-D^2).\]

\subsection*{Proof of Theorem \ref{main theorem}} The rest of this section is devoted to proving Theorem \ref{main theorem}.

\subsubsection{The complex $X_k$ does not PL embed in $\mathbb{S}^4$ for $k>1$}\label{noembedd}
Suppose that there is a PL embedding $f\colon X_k \rightarrow \mathbb{S}^4$ and that $k>1$. Again, the observation in Section \ref{simplify} applies. Also, as we saw in Section \ref{FKT}, the only nonzero linking numbers between pairs of disjoint $3$-cycles in $K_6$ or $\hat{K}_6$ are $\Lk(x_1x_2x_3, x_4x_5x_6) = \Lk(\x_1\x_2\x_3, \x_4\x_5\x_6) = 1$ mod $2$.

%\subsubsection*{Claim} Both $\Lk(x_1x_2x_3, x_4x_5x_6)$ and $\Lk(\x_1\x_2\x_3, \x_4\x_5\x_6)$ must be divisible by $k$. 

%\begin{proof}[Proof of claim]
%To see this, let $\alpha$ be the loop that is the singular set of $P_k$. Then $\beta\colon =\partial (P_k-D^2)$ is homotopic to $(k\cdot\alpha)$. So with respect to $f$, the homology cycle represented by $\beta$, which is an element of  $H_1(W;\mathbb{Z})$, is divisible by $k$. Now, $$[\beta] = [x_4x_5x_6] + [\x_4\x_5\x_6],$$
%and note that the inclusion $C\hookrightarrow W$ induces an isomorphism $H_1(W;\mathbb{Z}) \cong H_1(C; \mathbb{Z})$. Therefore, $[x_4x_5x_6] + [\x_4\x_5\x_6]$ is divisible in $ H_1(C; \mathbb{Z})$, which is isomorphic to $\mathbb{Z}^2 = \langle a\rangle \oplus \langle b\rangle$. (Recall that $a$ is represented by a small loop that goes around $x_1x_2x_3$, and $b$ is represented by a small loop that goes around $\x_1\x_2\x_3$.)

%Since $K_6$ and $\hat{K}_6$ lie in disjoint balls in $\mathbb{S}^3$, it follows that $[x_4x_5x_6] = r\cdot a$ and $[\x_4\x_5\x_6] = s\cdot b$. Therefore, in order for the sum to be divisible by $k$, we must have that both $r$ and $s$ are divisible by $k$. By noting that $r = \Lk(x_1x_2x_3, x_4x_5x_6)$ and $s = \Lk(\x_1\x_2\x_3, \x_4\x_5\x_6)$, we deduce the claim.
%\end{proof}
Let $\alpha$ be the loop that is the singular set of $P_k$ and let $\beta\colon =\partial (P_k-D^2)$. 
\begin{itemize}
\item[-] As an element in $\pi_1(W)$, the loop $f(\beta)$ is a $k$-th power since $\beta$ is homotopic to $k$ times $\alpha$.
\item[-] As an element in $\pi_1(C)$, the loop $f(\beta) = (x_4x_5x_6) * (\x_4\x_5\x_6),$
and since $K_6$ and $\hat{K}_6$ lie in disjoint balls in $\mathbb{S}^3$, it follows that 
\[ \eta (f(\beta)) = a^r\cdot b^s \;\in\; F_2,\]
for some $r, s = 1$ (mod $2$). In particular, $r, s \ne 0$, so $(a^r\cdot b^s)$ is not the $k$-th power in $F_2$ since $k>1$. \footnote{Hint: Look at how $(a^r\cdot b^s)$ acts on the Cayley graph of $F_2$ and its axis of translation, which goes through $1$ and $(a^r\cdot b^s)$. If $(a^r\cdot b^s) = g^k$ for some $k>1$, then the axis of translation of $g$ must coincide with that of $(a^r\cdot b^s)$. Then one should notice there is something strange about this situation.}
\end{itemize}
If $\pi_1(C)$ and $\pi_1(W)$ were isomorphic, then we would obtain a contradiction and be done, but they are not. However, recall that the inclusion $C\hookrightarrow W$ induces isomorphism on nilpotent quotients
\[\pi_1(W)/(\pi_1(W)_{n})\;\cong \;\pi_1(C)/(\pi_1(C)_{n}) \stackrel{\eta_n}\longrightarrow F_2/(F_2)_n = \langle a,b \rangle /(\langle a,b \rangle )_n, \]
for all $n \in \mathbb{Z}^+$. Thus, we would obtain a contradiction if there is a $n$ such that $(a^r\cdot b^s)$ is not a $k$-th power in $(F_2/(F_2)_n)$. This turns out to be the case for large enough $n$ by a theorem of Baumslag (\cite{Baumslag}).

\begin{theorem}[Baumslag, Proposition 2 in \cite{Baumslag}]
Let $p$ be a prime. Let $F$ be a free group and let $w$ be an element of $F$ which is not a $p$-th power. Then there is a finite $p$-group $G$ and a homomorphism $\theta\colon F \rightarrow G$ such that $\theta (w)$ is not a $p$-th power in $G$.
\end{theorem}

All $p$-groups are nilpotent. Therefore, $X_k$ does not PL embed in $\mathbb{S}^4$ for $k>1$.

\subsection*{A few words about words} We would like to make the following three remarks.

\begin{remark}
The splitting $\pi_1(C)\cong Z*\hat Z$ is crucial to our approach. If we merely knew that Lk$(x_1x_2x_3,\hat x_4\hat x_5\hat x_6)= \Lk(\hat x_1\hat x_2\hat x_3,x_4x_5x_6)=0$ and Lk$(x_1x_2x_3, x_4 x_5 x_6)= \Lk(\hat x_1\hat x_2\hat x_3,\x_4\x_5\x_6)= 1$ mod $2$ as in \cite{FKT}, then we would only be able to conclude that $f(\beta)$ is the image under a map $F_2\ra\pi_1(C)$ of a product of $r$ conjugates of $a$ followed by $s$ conjugates of $b$. This is not good enough since such a product can be a $k$-th power, e.g. for $k=3$,
$$[a(b^{-1}ab)(b^{-1}ab)]\;*\;[(aba^{-1})(a^2ba^{-2})(aba^{-1})]\; =\;(ab^{-1}a)(ab)^3(ab^{-1}a)^{-1}.$$   
\end{remark}

\begin{remark}
In the free group $F_2$ the product $a^3b^3$ fails to be a cube by a product of three commutators 
$$
a^3b^3\; =\; (ab)^3\; *\; (ab)^{-1}\left([(ab)^{-1},[b^{-1},a]] \;*
\;[b^{-1},a]\;*\;[b^{-2},a]\right)(ab),
$$
so if we replace $P_3$ by its connect sum with a genus $3$ surface $\Sigma_3$ then our method does not obstruct attempts to embed $X_3\#\Sigma_3$ (where the connect sum is made at an interior point of $P_3$) into $\mathbb S^4$. %In fact, one can use the above expression to show that $X_3\#\Sigma_3$ almost embeds in $\mathbb S^4.$
\end{remark}

\begin{remark}
We can be completely explicit about the group $G$ and the homomorphism $\theta$ in Baumslag's theorem for our word $a^rb^s$. Factor $r=p^im$ and $s=p^jn$ where $m,n$ are prime to $p$ and without loss of generality $i\leq j$. Take $G$ to be the semidirect product $(\mathbb Z/p)^{p^j}\rtimes\mathbb Z/p^{j+1}$ where $\mathbb Z/p^{j+1}$ acts by cyclically permuting the coordinates of $(\mathbb Z/p)^{p^j}$. (So multiples of $p^j$ in $\mathbb Z/p^{j+1}$ act trivially.) If we set
\begin{eqnarray*}
\theta(a)&=&((1,0,\dots,0),-p^{j-i}n),\\
\theta(b)&=&((0,\dots,0),m),
\end{eqnarray*}
then a direct computation shows that $\theta(a^rb^s)=(v,0)$ for a {\it non-zero} vector $v\in(\mathbb Z/p)^{p^j}$. Suppose this is a $p$-th power, say $(v,0)=(u,k)^p$. Then $kp=0$ modulo $p^{j+1}$ so $k$ is a multiple of $p^j$ and hence commutes with $u$ in the semidirect product. But then $(v,0)=(u,k)^p=(pu,pk)=(0,0)$ contradicting the fact that $v$ is non-zero. So $\theta(a^rb^s)$ is not a $p$-th power.
\end{remark}

\subsubsection{The van Kampen obstruction of $X_k$ is $0$ for odd $k$}\label{obstruction of X_k}

We will exhibit a PL immersion $f\colon X_k\rightarrow \mathbb{S}^4$ with van Kampen obstruction equal to zero when $k$ is odd. 

First, we PL embed $\bowtie$ into $\mathbb{S}^4$ by embedding  $K_6$ and $\hat{K}_6$ in disjoint balls in $\mathbb{S}^3$ using for each the embedding like the one in Figure 3(c) but with more ``twists" so that $$\Lk(x_1x_2x_3,x_4x_5x_6) = \Lk(\x_1\x_2\x_3,\x_4\x_5\x_6) =k,$$ and then connecting $x_6$ and $\x_6$ by a straight line segment outside these balls. Then we cone off $K_6$ to a point $f(x_0)$ and cone off $\hat{K}_6$ to a point $f(\x_0)$ in the Northern hemisphere in such a way that the two cones are disjoint. Then we fill in all $3$-cycles in $K_6$ and $\hat{K}_6$, except for $x_4x_5x_6$ and $\x_4\x_5\x_6$ with disks in the Southern hemisphere $\mathbb{D}^4$. The components of the van Kampen vector $V_f$ corresponding to the pairs of disjoint simplices in $\bowtie$ are clearly zero.

Next, in two steps we immerse into $\mathbb{D}^4$ the $(P_k-D^2)$, which we think of as a CW complex with one $2$-cell $B$ attached to its $1$-skeleton that is the dumbbell $\delta$ as in Figure 9. 
\begin{figure}[h!]
\centering
\includegraphics[scale=0.3]{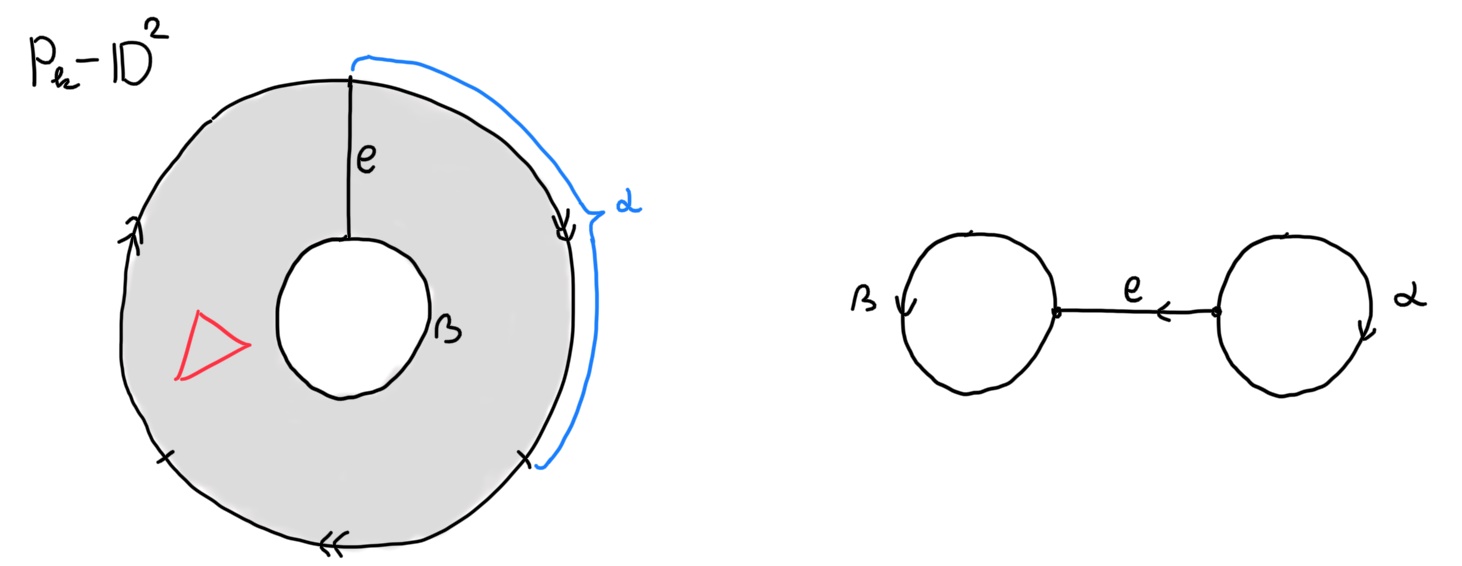}
\caption{Left picture: $P_k$ for $k=3$. Right picture: its $1$-skeleton $\delta$.}
\end{figure}
That is, $\delta = \alpha \cup \beta \cup e$. In the first step, we extend $f$ to the dumbbell $\delta$ by mapping it into $\mathbb{S}^3$ so that
\begin{itemize}
\item[.]  $\beta$ wraps around the loop $(x_4x_5x_6 * \x_4\x_5\x_6)$, 
\item[.] the edge $e$ is mapped to a point, and 
\item[.] $\alpha$ wraps around  $(a*b)^{-1}$, where $a$ is a small loop that goes around $x_1x_2x_3$ in the same direction as $x_4x_5x_6$, and $b$ is a small loop that goes around $\x_1\x_2\x_3$ in the same direction as $\x_4\x_5\x_6$. That is, $x_4x_5x_6 = ka $ and $\x_4\x_5\x_6 = kb$ in homology (and also fundamental group).
\end{itemize} 
Then in the second step, we extend $f$ to $B$ by filling in the image of its boundary $\partial B = \alpha^3*e*\beta*e^{-1}$ with a $2$-disk in $\mathbb{D}^4$. The map we have just constructed is not generic, but we can perturb it slightly to obtain a generic immersion. 

Now we compute the van Kampen vector $V_f$. Though $(P_k-D^2)$ is not a simplex, practically this causes no complication. (See the discussion ``Triangulations versus regular regions" in Section \ref{vKobstruction}.) We can designate the red triangle $T$ in Figure 9 to be the representative simplex of the regular region $(P_k-D^2)$ and we can assume that all intersections between $(P_k-D^2)$ and other simplices are contained in this triangle $T$. The boundary of $T$ is isotopic to $\partial B$ and we can harmlessly assume that the restriction of $f$ to $\partial T$ is very close to its restriction to $\partial B$. Now, we have
$$f(\partial B) = f(\alpha^k*e*\beta*e^{-1}) = (a*b)^{-k}*a^k*b^k,$$
which is trivial in $H_1(C) = H_1(\mathbb{S}^3 - x_1x_2x_3-\x_1\x_2\x_3)$. Therefore, the linking numbers $$\Lk(\partial B, x_1x_2x_3) = 0, \quad \Lk(\partial B, \x_1\x_2\x_3) = 0,$$
which implies that components of $V_f$ corresponding to the pair $(B,x_1x_2x_3)$ and $(B, \x_1\x_2\x_3)$ are equal to $0$. By construction $f(B)$ is disjoint from the images of all other $2$-simplices that are disjoint from $B$ in $X_k$. Therefore, $V_f = 0$, i.e. the van Kampen obstruction of $X_k$ is zero.

\begin{remark}
For the above map on $\delta$, it is not possible to map in $B$ into $\mathbb{D}^4$ to get an embedding of $X_k$. This is because $f(\partial B) =(a*b)^{-3}*a^3*b^3$ is nontrivial $\pi_1(W)$, which we can see in the $3$rd nilpotent quotient  $\pi_1(W)/(\pi_1(W))_3 \cong \pi_1(C)/(\pi_1(C))_3 \cong$\footnote{In this case, $\pi_1(C) \cong F_2 = \langle a,b \rangle$ since $f(x_1x_2x_3 \cup \x_1\x_2\x_3)$ is the standard trivial link in $\mathbb{S}^3$.} $F_2/(F_2)_3$, in which $[(ab)^{-3}a^3b^3]$ is clearly nonzero. This might suggest that the obstruction to $X_k$ PL embedding in $\mathbb{S}^4$ is in terms of triple linking numbers. This turns out not to be the case because we can go back and change the map $f$ on $\alpha$ to be $((a*b)^{-1}*[b^{-1},a])$ instead of $(a*b)^{-1}$. Then $f(\partial B)$ is zero in $\pi_1(W)/(\pi_1(W))_3$, so we have to look at it in $\pi_1(W)/(\pi_1(W))_4$ instead. This leads us to Section \ref{embeddings}.
\end{remark}

%\begin{theorem}[Baumslag]
%Let $w= w(a_1,a_2,..., a_n)$ be an element of the free group $F_n$ on $n$ generators $a_1,a_2, ..., a_n$. Suppose that $w$ is not a proper power or a primitive\footnote{An element of a free group is a \emph{primitive} if it can be included in a set of free generators.}. If $g_1,g_2, ..., g_n$ and $g$ are elements of a free group connected by the relation
%\[ w(g_1,g_2,.., g_n) = g^k \quad (k>1),\]
%then the rank of the group generated by $g_1, g_2,..., g_n$ and $g$ is at most $(n-1)$. 
%\end{theorem}

\section{Various immersions of $X_k$ (that hide the embedding obstructions)}\label{embeddings}
%\subsection*{A few words about how to show that a complex does not embed} 
%There are two approaches to showing that a $n$-complex does not embed in $\mathbb{S}^{2n}$. One is to start with a hypothetical embedding and arrive at a contradiction. The other is to start with a particular generic immersion which has some intersections and then show that the intersections do not all go away as one varies the immersion by cooking up an immersion invariant out of them. We used the former approach to show that our examples $X_k$ do not PL embed in $\mathbb{S}^4$. Part of the novelty of these examples is that, in contrast to the examples in \cite{FKT}, %and \cite{SSS},
% they not seem to be amenable to the later approach. We will  illustrate this in this section by exhibiting immersions that hide embedding obstructions from potential invariants. 

We will give, for each odd $k >1$, a family of PL immersions $$f_n\colon X_k\rightarrow \mathbb{S}^4,$$ for $n\in \mathbb{Z}^+$, that are similar to those in Section \ref{obstruction of X_k} but the ``obstruction" to $X_k$ embedding in $\mathbb{S}^4$ is zero in all nilpotent quotients of $\pi_1(W) = \pi_1(\mathbb{D}^4-x_1x_2x_3 - \x_1\x_2\x_3)$ of orders $1, 2,..., (n+1)$. 

\subsection{The immersions $f_n\colon X_k\rightarrow \mathbb{S}^4$}  First, we PL embed $\bowtie$ into $\mathbb{S}^4$ as follows. We use an embedding like the one given in Figure 3(c) but with more ``twists" to embed $K_6$ into $\mathbb{S}^3$ and with even a lot more ``twists" to embed $\hat{K}_6$ in a $3$-ball disjoint from $f(K_6)$. The number of ``twists" added should be so that the linking number 
\[\Lk(x_1x_2x_3, x_4x_5x_6) = k \quad \text{and}\quad \Lk(\x_1\x_2\x_3,\x_4\x_5\x_6) = k^{2^{n-1}}.\]
Then we connect $x_6$ and $\x_6$ by a straight line segment outside these two balls and define $f$ on $\bowtie$ in the exact same way as in Section \ref{obstruction of X_k}. 

Next, we define $f_n$ on $\delta = \alpha \cup \beta \cup e$ as follows.
\begin{itemize}
\item[.]  $\beta$ wraps around the loop $(x_4x_5x_6 * \x_4\x_5\x_6)$, 
\item[.] the edge $e$ is mapped to a point, and 
\item[.] $\alpha$ wraps around  $w(a,b)$, where $a$ is a small loop that goes around $x_1x_2x_3$ in the same direction as $x_4x_5x_6$, $b$ is a small loop that goes around $\x_1\x_2\x_3$ in the same direction as $\x_4\x_5\x_6$, and $w(a,b)$ is a word in $a,b$ that is to be determined. 
\end{itemize} 
Then we extend $f_n$ to $B$ by filling in the image of its boundary $\partial B = \alpha^k*e*\beta*e^{-1}$ with a $2$-disk in $\mathbb{D}^4$. The map constructed is not generic, but we can perturb it to obtain a generic immersion. 

In $\pi_1(W)$, we have
\[f_n(\partial B) = f_n(\alpha^k*e*\beta*e^{-1}) =  f_n(\alpha)^k* f_n(\beta ) = f_n(\alpha)^k*(a^k*b^{k^{2^{n-1}}}).\]
Thus, in order for $f_n(\partial B)$ to be trivial in $\pi_1(W)/(\pi_1(W))_{n+1} \cong \pi_1(C)/(\pi_1(C))_{n+1} \cong F_2/(F_2)_{n+1}$, the element $(a^kb^{k^{2^{n-1}}})$ must be a $k$-th power in $F_2/(F_2)_{n+1}$. We need the following proposition, which we will prove afterwards. 

%\subsection{Products of $p$-th powers in nilpotent groups}
\begin{proposition}\label{pth powers}
Let $p$ be a positive integer and let $G$ be an $n$-step nilpotent group. For any $a,b\in G$, $a^pb^{p^{2^{n-1}}}$ is a $p$-th power.  
\end{proposition}

Since $F_2/(F_2)_{n+1}$ is an $n$-step nilpotent group, Proposition \ref{pth powers} applies to give an element $w(a,b)$ such that $$a^kb^{k^{2^{n-1}}} = \;w(a,b)^{-k}$$ 
in $F_2/(F_2)_{n+1}$. We will use this $w(a,b)$ to define $f_n$ on $\alpha$. Then $f_n(\partial B)=0$  in $\pi_1(W)/(\pi_1(W))_{n+1}$. %That is, all the Milnor invariants of orders up till $n$ for the triple $(x_1x_2x_3, \x_1\x_2\x_3, \partial B)$ vanish. 

Lastly, note that any other triple of disjoint $3$-cycles in $(K_6\cup \hat{K}_6)$ (that bound a disk in $X_k$) must have one of the cycles lying in a ball disjoint from the other two, which (by inspection of Figure 3(c)) are isotopic to the unlink. Therefore, each one of them must be zero in any nilpotent quotient of the complement of the other two in $\mathbb{S}^3$.

\subsection{Proof of Proposition \ref{pth powers}}
Recall that for a group $G$, we write $G_1=G, G_{i+1}=[G,G_i]$. Let $G_i^r$ be the subgroup of $G_i$ generated by the $r$-th powers.

\begin{lemma}
 $G_{n-i-1}^{p^{2^{i+1}-1}}$ is in the center of $G/G_n^p\cdots G_{n-i}^{p^{2^i}}$. 
\end{lemma}
\begin{proof}
The proof is by induction. To carry it out, we prove the more general statement 
\begin{itemize}
\item
$G_{k-i}$ and $G_{n-k}^{p^{2^{i+1}-1}}$ commute in $G/G_n^p\cdots G_{n-i}^{p^{2^i}}$.
\end{itemize}
The lemma is the special case when $k=i+1$, but the induction also uses larger $k$.  

For the proof it is useful to recall that the commutator of $x$ and $y^m$ can be rewritten as 
$$
[x,y^m]=xy^mx^{-1}y^{-m}=(xyx^{-1})^my^{-m}=([x,y]y)^my^{-m}.
$$ 
For the base case, note that for any $x\in G_k$ and $y\in G_{n-k}$ the commutator $[x,y]\in G_n$ is in the center of $G$, and therefore $[x,y^p]=([x,y]y)^py^{-p}=[x,y]^p$. So, $G_k$ commutes with $G_{n-k}^p$ in $G/G_n^p$. 

%In other words, $y^3$ commutes with $G_k$ in $G/G_n^3$. Since $G^3_{n-k}$ is generated by such cubes, we conclude that $G^3_{n-k}$ commutes with $G_k$ in $G/G_n^3$. 

Now, suppose we have proved the lemma for $i-1$ and want to prove it for $i$. To reduce notation, set $*=p^{2^i-1}$, so the $i-1$ statement says that $G_{k-(i-1)}$ commutes with $G^*_{n-k}$ in $G/G_n^p\cdots G_{n-(i-1)}^{p^{2^{i-1}}}$. Now, we pick $x\in G_{k-i}$, $y\in G_{n-k}$ and compute 
\begin{eqnarray*}
[x,y^{p**}]&=&([x,y^*]y^*)^{p*}y^{-p**}\\
&=&[x,y^*]^{p*} \hspace{2cm}\mbox{ in } G/G_n^p\cdots G_{n-(i-1)}^{p^{2^{i-1}}}\\
&=&1 \hspace{3.1cm}\mbox{ in } G/G_n^p\cdots G_{n-(i-1)}^{p^{2^{i-1}}}G_{n-i}^{p^{2^i}}.
\end{eqnarray*}
The first line follows from the commutator formula, the second line from the inductive hypothesis since $[x,y^{*}]\in G_{k-(i-1)}$ and $y^*\in G^*_{n-k}$, and the third line from $[x,y^{*}]\in G_{n-i}$ since $p*=p^{2^i}$. 

In summary, we have shown that $G_{k-i}$ commutes with $G_{n-k}^{p**}=G_{n-k}^{p^{2^{i+1}-1}}$ in the relevant quotient, which finishes the inductive step and the proof. 
\end{proof}
The lemma implies the image of $G^{p^{2^{i+1}}}_{n-i-1}$ in $G/G_n^p\cdots G_{n-i}^{p^{2^i}}$ consists of $p$-th powers of central elements. 
\begin{corollary}
If $x$ is a $p$-th power in $G/G_n^p\cdots G_{n-i}^{p^{2^i}}G_{n-(i+1)}^{p^{2^{i+1}}}$, then it is a $p$-th power in $G/G_n^p\cdots G_{n-i}^{p^{2^i}}$. 
\end{corollary}
\begin{proof}
We compute in the quotient $G/G_n^p\cdots G_{n-i}^{p^{2^i}}$. 
By assumption, $x=u^pw$ for $w\in G^{p^{2^{i+1}}}_{n-i-1}$. By the previous lemma, this $w$ is the $p$-th power of a central element $v$, i.e. $w=v^p$. Therefore $x=u^pv^p=(uv)^p$. 
\end{proof}
Clearly $a^pb^{p^{2^{n-1}}}$ is a $p$-th power in $G/G_n^p\cdots G_1^{p^{2^{n-1}}}$. Repeatedly applying the previous corollary shows that it is a $p$-th power in $G$. This finishes the proof of the proposition.

\section{Embedding obstructions for octahedralizations}\label{octahedral}
\subsection{Octahedralizations and some contexts}
Let $L$ be a $d$-dimensional simplicial complex. The {\it octahedralization} $OL$ is a complex obtained from $L$ as follows. For each vertex $v$ of $L$ there are two vertices $v$ and $\underline{v}$ in $OL$. For each $k$-simplex $v_0*\dots*v_k$ in $L$, there are $2^{k+1}$ $k$-simplices in $OL$ of the form $v_0^**\dots*v_k^*$, where $v_i^*\in\{v_i,\underline{v_i}\}$. Note that $OL$ depends on the triangulation of $L$ and not just on its topology.

In \cite{ADOS} it was computed that, for a flag triangulation of $L$, the van Kampen obstruction to $PL$-embedding $OL$ in $\mathbb S^{2d}$ vanishes if and only if $H_d(L;\mathbb Z/2)=0$. For $d\geq 3$, it is the only obstruction to finding such an embedding, so in that case the octahedralization of any $\mathbb Z/2$-acyclic, flag triangulated $d$-complex PL embeds in $\mathbb S^{2d}$. In \cite{AOS} this observation was combined with the Davis reflection group trick to construct, for $d\geq 3$, closed aspherical $(2d+1)$-manifolds whose $\mathbb F_p$-homology grows linearly in a residual sequence of regular covers for a fixed odd prime $p$. This should be contrasted with a conjecture of Singer which says for odd dimensional closed aspherical manifolds that the $\mathbb Q$-homology cannot grow linearly in such a sequence of covers. 

To obtain such manifolds in dimension $5$, one would need to $PL$ embed $OL$ in $\mathbb S^4$ for some flag %footnote saying what this means? 
triangulation of a $2$-complex $L$ with $H_2(L;\mathbb Z/2)=0$ but $H_2(L;\mathbb Z/p)\not=0$. The simplest such $L$ are the pseudo-projective planes $P_p$. This leads to the following question raised in the remark at the end of \cite{AOS}.

\begin{question}
Fix an odd prime $p$ and a flag triangulation of the pseudo-projective plane $P_p$. Does its octahedralization $OP_p$ embed PL in $\mathbb S^4$? 
\end{question} 

In this section, we will show that for a sufficiently fine triangulation of $P_k$, the octahedralization $OP_k$ does not $PL$-embed in $\mathbb S^4$. More precisely, we prove %The idea of the proof is to find a complex similar to $X_k$ inside such an octahedralization. 

\begin{proposition}\label{octanotembed}
Pick a flag triangulation of $P_k$ that contains two vertices $v,\hat v$ whose closed $2$-neighborhoods $N_2(v),N_2(\hat v)$ are disjoint from each other and from the singular set $\alpha$ of $P_k$. For this triangulation, $OP_k$ does not $PL$-embed in $\mathbb S^4$. 
\end{proposition}

\subsection{Proof of Proposition \ref{octanotembed}}
Suppose it does. Pick a path $\rho$ in the regular part of $P_k$ connecting $v$ with $\hat v$ as in Figure 10.
\begin{figure}[h!]
\label{neighborhood}
\centering
\includegraphics[scale=0.35]{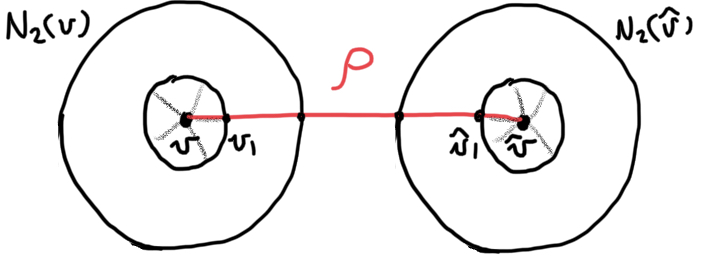}
\caption{Two $2$-neighborhoods connected by a path in the regular region of $P_k$.}
\end{figure}

% The role of $\bowtie$ will be played by $ON_2(v)\cup \rho\cup ON_2(\hat v)$. 
Arguing as in Section \ref{simplify}, any $PL$ embedding of $OP_k$ into $\mathbb S^4$ induces a $PL$ embedding $f$ into $\mathbb D^4$ of the subcomplex of $OP_k$ consisting of those simplices not containing $v$ or $\hat v$. Moreover, the intersection of $f$ with $\partial\mathbb D^4=\mathbb S^3$ consists of $O(Lk(v))$ embedded in $\mathbb D^3$, $O(Lk(\hat v))$ embedded in $\hat{\mathbb D}^3$ and a standard edge $v_1\hat v_1\subset\rho$ connecting them.

\begin{lemma}
The graph $O(Lk(v))$ is not linklessly embeddable.
\end{lemma}
\begin{proof} 
Since the triangulation is a flag triangulation, $Lk(v)$ is an $n$-gon for some $n\geq 4$. When $n=4$ the octahedralization $OLk(v)$ is the graph $K_{4,4}$, which is not linklessly embeddable by a theorem of Sachs (\cite{sachs}). 
\begin{figure}[h!]
\label{octa}
\centering
\includegraphics[scale=0.3]{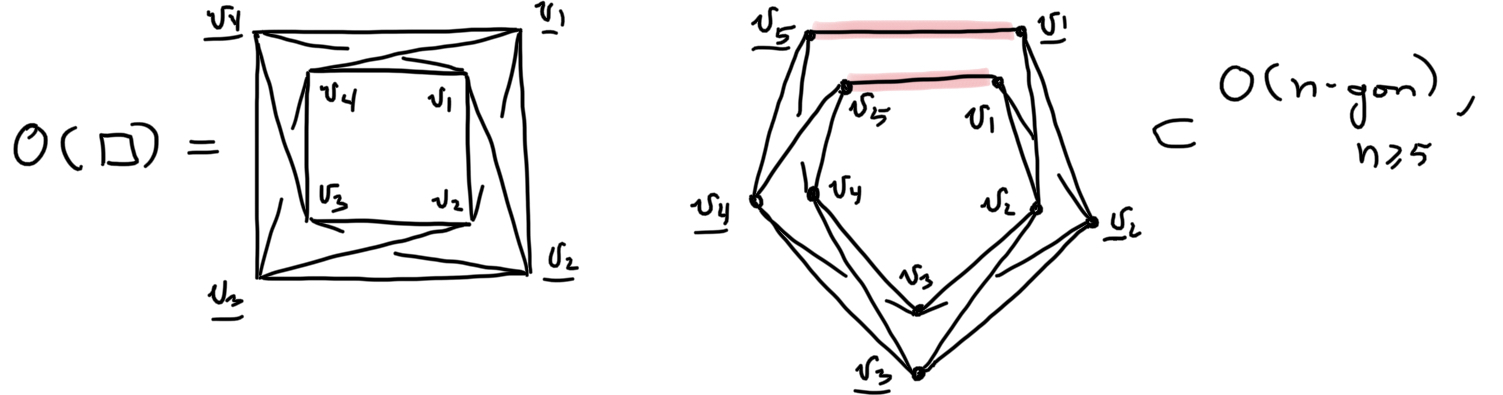}
\caption{Octahedralizations of links.}
\end{figure}
When $n>4$ the octahedralization is larger, but it contains the subgraph on the right of Figure 11. This subgraph is also not linklessly embeddable since we can collapse it to $K_{4,4}$ by collapsing two edges.
\end{proof}

 %It comes out of the proof that the reason these graphs are not linklessly embeddable is that there are two disjoint loops $\sigma$ and $\tau$ in $O(Lk(v))$ such that $f(\sigma)$ and $f(\tau)$ have linking number $1$ mod $2$.
So, we can pick disjoint loops $\sigma,\tau,\in O(Lk(v))$, $\hat\sigma,\hat\tau\in O(Lk(\hat v))$ so that $f(\sigma)$ links $f(\tau)$ and $f(\hat\sigma)$ links $f(\hat\tau)$. Moreover, without loss of generality, $v_1\in\sigma$ and $\hat v_1\in\hat\sigma$. 

Now, let $C:=\mathbb S^3-f(\tau\cup\hat\tau)$, $W:=\mathbb D^4-f(\underline {v}*\tau\cup\underline{\hat v}*\hat\tau)$, and $\beta:=\sigma*(v_1\hat v_1)*\hat\sigma*(\hat v_1v_1)$. We will show next that $f(\beta)$ is a $k$-th power in $\pi_1(W)$. Given this, proceeding as in Section \ref{proofoftheorem} proves the proposition. We need the following lemma.
\begin{lemma}
\label{onto}
Suppose $\sigma$ and $\tau$ are disjoint loops in $O(Lk(v))$ such that $f(\sigma)$ and $f(\tau)$ link in $\mathbb S^3$. Then each of $\sigma$ and $\tau$ projects bijectively onto the loop $\partial(N_1(v))$.
\end{lemma}

\begin{figure}[h!]
\label{fill}
\centering
\includegraphics[scale=0.65]{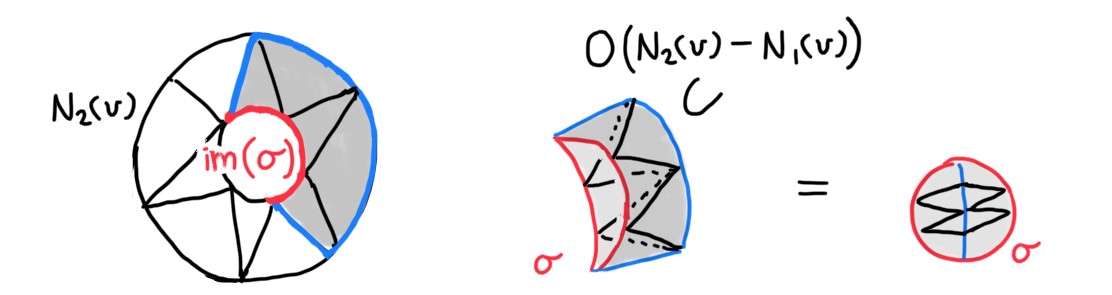}
\caption{Filling $\sigma$ in the octahedralization of the annulus $N_2(v)-N_1(v)$.}
\end{figure}

\begin{proof}
If one of the loops, say $\sigma$, does not map onto $Lk(v)$, then it bounds an embedded disk in $O(N_2(v)-N_1(v))$ (see Figure 12). Since we can fill the other loop $\tau$ with the disk $\underline v*\tau$, the images of $f(\sigma)$ and $f(\tau)$ bound disjoint embedded disks in $\mathbb D^4$, so $f(\sigma)$ and $f(\tau)$ do not link in $\mathbb S^3$, contradicting the assumption. So, both $\sigma$ and $\tau$ map onto. Since they are disjoint and $O(Lk(v))$ has twice as many vertices as $Lk(v)$, this is only possible if $\sigma$ and $\tau$ map bijectively onto $Lk(v)$. 
\end{proof}
The lemma implies that, after interchanging some vertices $w$ with $\underline w$ where necessary, we can assume that $\sigma=\partial N_1(v), \tau=\underline{\partial N_1(v)}$ and similarly $\hat\sigma=\partial N_1(\hat v),\hat\tau=\underline{\partial N_1(\hat v)}$. Since $\beta=\partial N_1(v)*(v_1\hat v_1)*\partial N_1(\hat v)*(\hat v_1v_1)$ is homotopic to $k$ times $\alpha$ in $P_k-N_1(v\cup\hat v)$ and $f(P_k-N_1(v\cup\hat v))$ is contained in $W$, we conclude that $f(\beta)$ is a $k$-th power in $\pi_1(W)$. So, we have proved Proposition \ref{octanotembed}.
\newline

%\begin{proof}
%If one of the loops, say $\sigma$, does not map onto $Lk(v)$, then it bounds in $O(N_2(v)-N_1(v))$ (see Figure 12). Since we can fill $\tau$ with $\underline v*\tau$, this shows that $f(\sigma)$ and $f(\tau)$ do not link, contadicting the assumption. So, both $\sigma$ and $\tau$ map onto. Since they are disjoint and $O(Lk(v))$ has twice as many vertices as $Lk(v)$, this is only possible if $\sigma$ and $\tau$ map bijectively onto $Lk(v)$. 
%\end{proof}
%The lemma implies that, after interchanging some vertices $w$ with $\underline w$ where necessary, we can assume that $\sigma=\partial N_1(v), \tau=\underline{\partial N_1(v)}$ and similarly $\hat\sigma=\partial N_1(\hat v),\hat\tau=\underline{\partial N_1(\hat v)}$. Since $\partial N_1(v)*(v_1\hat v_1)*\partial N_1(\hat v)*(\hat v_1v_1)$ is homotopic to $k$ times $\alpha$ in $P_k-N_1(v\cup\hat v)$, it follows that $\beta$ is homotopic to $k$ times $\alpha$ in $O(P_k-N_1(v\cup\hat v))$ and therefore $f(\beta)$ is a $k$-th power in $\pi_1(W)$. So, we have proved Proposition \ref{octanotembed}.
%\newline

Lastly, for the readers who have read the remarks at the end of Section \ref{proofoftheorem}, we will leave you with the following question.

\begin{question}
Does any of the octahedralizations $O(P_k \# \Sigma_g)$ embed PL in $\mathbb{S}^4$? 
\end{question}

\section{Embedding obstructions formulated by Krushkal (\cite{Krushkal}) vanish on $X_k$ }\label{Krushkalobstruction}
In \cite{Krushkal}, Krushkal used $4$-dimensional thickenings of $2$-complexes to define obstructions to embedding $2$-complexes in $\mathbb R^4$ that detect the non-embedding of the Freedmann-Krushkal-Teichner (abbreviated FKT below) examples. In this section we will briefly summarize Krushkal's obstructions and show that they vanish for our examples $X_k$. 
\subsection*{Thickenings}
The $1$-skeleton of a $2$-complex $K$ can be thickened in a unique way to a $4$-manifold with boundary $(V,\partial V)$. This $4$-manifold has a $0$-handle for each vertex and a $1$-handle for each edge. It deformation retracts onto an embedded copy $K^{(1)}\hookrightarrow V$ of the $1$-skeleton of the $2$-complex. To obtain a {\it thickening} of the $2$-complex $K$ from this, one attaches a $2$-handle for each $2$-cell along an attaching loop in $\partial V$ that is homotopic in $V$ to the attaching map of the $2$-cell in $K^{(1)}$. The resulting $4$-manifold $M$ is homotopy equivalent to $K$ via a collapse map $M\ra K$. There can be different valid choices for the attaching loops in $\partial V$ and as a result a $2$-complex $K$ can have many different $4$-dimensional thickenings.

\begin{example}
Both the pair of pants $A= (S^2 -\{D^2\cup D^2\cup D^2\})$ and the torus with a disk removed $B= (T^2 - D^2)$ are $2$-dimensional thickenings of the figure $8$ graph, so $A\times A$ and $B\times B$ are $4$-dimensional thickenings of the $2$-complex $8\times 8$. The thickenings $A\times A$ and $B\times B$ can be distinguished by the intersection form on $H_2(\cdot;\mathbb Q)$, since the intersection form on $A\times A$ vanishes, while the intersection form on $B\times B$ does not. 
\end{example} 

%To relate thickenings with immersions of $2$-complexes, note that for any PL immersion of $K$ into $S^4$, pulling back the regular neighborhood gives a thickening of $K$. 

\subsection{Krushkal's obstructions}
For a PL embedding of $K$ into $\mathbb S^4$ one can take the regular neighborhood $M$ of $K$ to get a thickening. Since this thickening is embedded in $\mathbb S^4$, it has special properties. For instance, the intersection form on $H_2(M;\mathbb Q)$ vanishes, since intersection numbers between homology cycles in $M$ can be computed in $\mathbb S^4$ where all these intersection numbers are zero. Krushkal observed that if $H_1(K;\mathbb Q)=0$, then there are additional restrictions on thickenings $M$ that embed in $\mathbb S^4$, expressed in terms of the cohomology of the boundary $\partial M$, as follows.

\begin{theorem}[Lemma 3.6 in \cite{Krushkal}]
Suppose that $M$ is a $4$-manifold with boundary that embeds in $\mathbb S^4$, and $H_1(M;\mathbb Q)=0$. Then all Massey products on $H^1(\partial M;\mathbb Q)$ vanish. 
\end{theorem}
Thus, if a $2$-complex $K$ embeds PL in $\mathbb{S}^4$, then it admits a thickening $M$ whose intersection form on $H_2(M;\mathbb Q)$ and Massey products on $H^1(\partial M;\mathbb Q)$ all vanish. Therefore, one way to show that $K$ does not PL embed in $\mathbb{S}^4$ is to look at \emph{all} thickenings $M$ of $K$ and show that for \emph{each one of them} either the intersection form on $H_2(M;\mathbb Q)$ or one of the Massey products on $H^1(\partial M;\mathbb Q)$ does not vanish. These are the obstructions introduced in \cite{Krushkal}.

To look at all thickenings of a $2$-complex might sound like a daunting task. Nonetheless, Krushkal showed that for {\it any} thickening of an FKT example, one of these obstructions is non-zero by making use of Conway-Gordon's theorem and the relation between Massey products on $\partial M$ and Milnor invariants of links in $\mathbb{S}^3$ given by a theorem of Turaev. In general, however, we do not know how one can compute the above obstructions for a $2$-complex (except for our examples below).

\subsection{Krushkal's obstructions of $X_k$ are zero}
Let us now turn to our examples. Recall that $X_k=\Delta^2_6\#\hat\Delta^2_6\#P_k$. % where $\Delta^2_6$ and $\hat\Delta^2_6$ are copies of the $2$-skeleton of the $6$-dim simplex and $P_k$ is a pseudo-projective plane obtained by gluing $D^2$ to a circle via a degree $k$ map. 
Note that $\pi_1(X_k) = \mathbb Z/k$, so that $H_1(X_k;\mathbb Q)=0$ and we are in a situation where it makes sense to investigate Krushkal's obstructions. Denote by $\alpha_k$ the singular set of $P_k$.

We will show that Krushkal's obstructions of $X_k$ are zero by finding thickenings of $X_k$, which we will call $M_k$, for which the intersection form on $H_2(M_k;\mathbb{Q})$ and all Massey products on $H^1(\partial M_k; \mathbb{Q})$ vanish. We will first show that the complex $X_1$ embeds PL in $\mathbb{S}^4$. Therefore, $X_1$ has a thickening whose intersection form and Massey products vanish. Then we show that a thickening of $X_k$ can be obtained from a thickening of $X_1$ via a process that does not change either the intersection form or Massey products.

\begin{lemma}
The complex $X_1$ embeds PL in $\mathbb S^4$.
\end{lemma}
\begin{proof}
Van Kampen showed in \cite{vanKampenpaper} that the $2$-complex $\Delta_6^2$ immerses PL in $\mathbb S^4$ with a single self-intersection, so removing a small $2$-disk from the interior of a cell in $\Delta_6^2$ yields a $2$-complex $(\Delta_6^2-{D^2})$ that does embed in $\mathbb S^4$. Now, pick two disjoint embeddings of this complex $(\Delta_6^2- D^2)\hookrightarrow \mathbb S^4$ and $\hat (\Delta_6^2-\hat D^2)\hookrightarrow \mathbb S^4$. Use a path from $\partial D^2$ to $\partial\hat D^2$ (embedded and with interior disjoint from both of the $2$-complexes) to guide an isotopy of the first embedding which drags a small neighborhood of an interval $I$ in $\partial D^2$ along the path until it is glued to an interval $\hat I$ in $\partial\hat D^2$. The resulting $2$-complex $(\Delta^2_6- D^2)\cup_{I=\hat I}(\hat\Delta^2_6- \hat D^2)$ embedded in $\mathbb S^4$ is precisely $X_1$. 
%Use a path from $\partial D^2$ to $\partial\hat D^2$ (embedded and with interior disjoint from both of the $2$-complexes) to modify the first embedding so that an interval $I$ in $\partial D^2$ is glued to an interval $\hat I$ in $\partial\hat D^2$. The resulting $2$-complex $(\Delta^2_6- D^2)\cup_{I=\hat I}(\hat\Delta^2_6- \hat D^2)$ embedded in $\mathbb S^4$ is precisely $X_1$. 
\end{proof} 

\subsection*{The thickening $M_k$}
When $k=1$, the regular neighborhood of a PL embedding $X_1\hookrightarrow \mathbb S^4$ is a thickening of $X_1$, which we will call $M_1$. For other values of $k$, since $$X_k = X_1 \cup_{\alpha_1} (P_k-D^2),$$
to build a thickening of $X_k$ we will assemble a thickening of $X_1$ and a thickening of $(P_k-D^2)$ together. Note that $(P_k-D^2)$ can be thought of as a bundle of $k$-pods over $\mathbb{S}^1$ with monodromy given by a $(2\pi/k)$-rotation, so it embeds in $\mathbb{S}^3$ and it is easy to see that a regular neighborhood in $\mathbb S^4$ of such an embedding of    $(P_k-D^2)$ is $S^1\times D^3$. The thickening of $X_1$ we will use is $M_1$, and the thickening of $(P_k-D^2)$ we will use is $S^1\times D^3$. We assemble these two pieces as follows. 
 
The boundary $\partial M_1$ of $M_1$ contains a nearby copy of $\alpha_1$, which we call $\alpha_1'$. Let $T$ be a regular neighborhood of this curve $\alpha_1'$ in $\partial M_1$, so $T$ is a $3$-dimensional solid torus. %(i.e. it is homeomorphic to $S^1\times D^2$)
The thickening $M_k$ of $X_k$ will be obtained by attaching to $M_1$ a copy of $S^1\times D^3$, identifying the solid torus $T$ in $\partial M_1$ with another solid torus in the boundary of $S^1\times D^3$. Now, look at the boundary $\partial(S^1\times D^3)=S^1\times S^2$. Let %$s\in S^2$ be the south pole and 
$S^1\subset S^2$ be the equator. Now, with an abuse of notation, %let $\alpha_k=S^1\times s$, and 
let $\alpha_1'$ be the loop that is a $(k,1)$-torus knot in $S^1\times S^1\subset S^1\times S^2$ (i.e. the loop $\alpha_1'$ goes around the first $S^1$-factor $k$ times). This is the loop that we will identify with the loop $\alpha_1' \subset \partial M_1$. Now, identify a regular neighborhood of $\alpha_1'$ in $S^1\times S^2$ with the solid torus $T\subset \partial M_1$. We obtain
$$
M_k:=M_1\cup_{T}(S^1\times D^3)
$$ 
the desired thickening of $X_k$.
%\begin{lemma}
%$M_k$ is a thickening of $X_k$. 
%\end{lemma}
\begin{remark}
Alternatively, one can extend the above PL embedding of $X_1$ to a PL immersion of $X_k$ and pull back the regular neighborhood of the image to obtain the same thickening $M_k$ of $X_k$.
\end{remark}
\subsection*{Comparing $M_1$ and $M_k$}
It follows from Mayer-Vietoris that the inclusion $M_1\hookrightarrow M_k$ is a rational homology (and cohomology) isomorphism. In particular, since the intersection form on $M_1$ vanishes, the intersection form on $M_k$ vanishes as well. Now, let $\partial M_1\times [0,1]$ be an inward collar neighborhood of the boundary $\partial M_1$ in $M_1$. Cutting out its complement $M_1-(\partial M_1\times[0,1])$ from $M_k$ leaves a cobordism  
$$W:=(\partial M_1\times [0,1])\cup_{T\times\{1\}}(S^1\times D^3),$$ between $\partial M_1$ and $\partial M_k$. Since $M_1\hookrightarrow M_k$ is a rational homology (and cohomology) isomorphism, excision implies that $i_1:\partial M_1\times[0,1]\hookrightarrow W$ is a rational homology (and cohomology) isomorphism, and Poincare-Lefschetz duality $H_*(W,\partial M_1;\mathbb Q)\cong H^{4-*}(W,\partial M_k;\mathbb Q)$ implies that $i_k:\partial M_k\hookrightarrow W$ is, as well. 

Since $M_1$ embeds in $\mathbb S^4$, all Massey products on $\partial M_1$ vanish. Now, given a rational cohomology isomorphism $f:X\ra Y$, vanishing of all Massey products is a property that holds for $X$ if and only if it holds for $Y$ (see \cite{kraines}). So, the cohomology isomorpism $i_1$ shows that all Massey products on $W$ vanish and from there the cohomology isomorpism $i_k$ shows that all Massey products on $\partial M_k$ vanish. 

In conclusion, we have constructed a thickening $M_k$ of $X_k$ for which all of Krushkal's obstructions vanish.

\bibliography{Reference}
\bibliographystyle{amsplain}

\end{document}